\documentclass{article}

\usepackage{geometry}
\usepackage[mathcal]{euscript}
\usepackage{hyperref}
\usepackage{amsmath,amssymb,amsthm}
\usepackage{mathtools}
\usepackage{tikz}\usetikzlibrary{patterns}
\usepackage{dutchcal}

\DeclareMathAlphabet{\mathpzc}{OT1}{pzc}{m}{it}

\usepackage{xcolor}

\newtheorem{thm}{Theorem}[section]
\newtheorem{lem}[thm]{Lemma}
\newtheorem{prop}[thm]{Proposition}

\theoremstyle{remark}
\newtheorem{rem}[thm]{Remark}

\theoremstyle{definition}

\newtheoremstyle{Claim}{}{}{\itshape}{}{\itshape\bfseries}{:}{ }{#1}
\theoremstyle{Claim}

\newcommand{\T}{{\mathbb{T}^N}}

\newcommand{\R}{\mathbb{R}}
\newcommand{\RsetN}{\mathbb{R}^N}

\newcommand{\Fc}{\mathbcal{F}}
\newcommand{\Gc}{\mathbcal{G}}
\newcommand{\Lc}{\mathbcal{L}}
\newcommand{\su}{\mathpzc{u}}
\newcommand{\sm}{\mathpzc{m}}

\newcommand{\eps}{\varepsilon}

\DeclareMathOperator{\diverg}{div}

\titlepage
\title{On the existence of oscillating solutions \\ in non-monotone Mean-Field Games}
\author{Marco Cirant}
\date{\today}

\begin{document}

\maketitle

\begin{abstract} For non-monotone single and two-populations time-dependent Mean-Field Game systems we obtain the existence of an infinite number of branches of non-trivial solutions. These non-trivial solutions are in particular shown to exhibit an oscillatory behaviour when they are close to the trivial (constant) one. The existence of such branches is derived using local and global bifurcation methods, that rely on the analysis of eigenfunction expansions of solutions to the associated linearized problem. Numerical analysis is performed on two different models to observe the oscillatory behaviour of solutions predicted by bifurcation theory, and to study further properties of branches far away from bifurcation points. 
\end{abstract}

\noindent
{\footnotesize \textbf{AMS-Subject Classification}}. {\footnotesize 35K55, 35B32, 35B36, 49N70.}

\section{Introduction}

We consider a system of partial differential equations arising in finite-horizon Mean-Field Games (briefly, MFG) with two populations of agents of the form
\begin{equation}\label{MFG0}
\begin{cases}
-\partial_t u_i - \sigma \Delta u_i + \frac{1}{2}|\nabla u_i|^2 = V_i(m_1, m_2), \\
\partial_t m_i - \sigma \Delta m_i - \diverg(\nabla u_i\, m_i)  = 0, & \text{in $Q_T = \Omega \times (0,T), i =1, 2,$}
\end{cases}
\end{equation}
endowed with Neumann boundary data
\[
\partial_\nu u_i = \partial_\nu m_i = 0 \quad \text{on $\partial \Omega \times (0,T), i =1, 2,$} 
\]
and initial-final conditions
\begin{equation}\label{ifcond}
m_i(x,0) \equiv |\Omega|^{-1}, \quad u_i(x,T) \equiv 0 \quad \text{on $\Omega, i =1, 2$}.
\end{equation}
Here, $\sigma > 0$, $\Omega$ is a smooth bounded domain of $\RsetN$, and the parameter $T > 0$ is the horizon of the game; the unknown $m(t) = (m_1(t), m_2(t))$ is a vector of probability densities on $\Omega$ and represents the evolution of the distributions of typical agents, $V_i$ is the function associated to their running costs, and $u$ is the vector of their value functions. MFG systems of PDEs have been introduced in the pioneering works \cite{LL062, LL07} to describe Nash equilibria of games with an infinite number of identical agents; we refer to \cite{BenFr, CardaNotes, GomesBook, HCM06, Lcol} and references therein for additional details on the theory of MFG. If players belong to two or more different groups (populations), then one is naturally led to consider MFG systems of the form \eqref{MFG0}. The systematic study of this setting started with the early works \cite{C14, feleqi, LachapelleWolfram}, see also \cite{Lauriere, Carda17, GoPa}.

Up to now, uniqueness of solutions to \eqref{MFG0} is known in two different regimes. The first one has been discussed in \cite{C14}; it is a straightforward generalization of the monotonicity condition by Lasry-Lions \cite{LL07}, and reads
\begin{equation}\label{uniquecond}
 \sum_{i=1, 2} \int_\Omega \big(V_i(m(x))-V_i (\bar m(x))\big) (m_i(x) - \bar{m}_i(x) )dx \ge 0 \quad \forall m, \bar{m} \in C(\T).
\end{equation}
Heuristically, this condition imposes aversion to crowd within each population, and this effect to be dominant with respect to effects due to interactions between different populations. Another uniqueness regime was discussed in \cite{Lcol}, and has been recently revived in \cite{BCir, BarFis, CGM}: it occurs under the ``smallness'' of some data. A typical example of this case is that the time horizon $T$ be small enough. Roughly speaking, if the horizon is short agents do not have enough time to reach an equilibrium that is far from their initial state. Both regimes turn out to be quite special in many applications. First, the main cause of dynamics in multi-population models is indeed interaction between populations; secondly, many interesting phenomena are unveiled if the system has enough time to evolve, showing in the long run its own typical features.

In this paper, we consider \eqref{MFG0} in a ``non-monotone'' case, namely without \eqref{uniquecond} in force. The purpose of this work is to prove first that without \eqref{uniquecond}, the MFG system \eqref{MFG0} admits in general multiple solutions. Secondly, that for time horizons $T^\eps$, which can be arbitrarily large, there are families of (small amplitude) solutions $(u_i^\eps, m_i^\eps)$ that exhibit an {\it oscillatory} behaviour in time, namely such that $m^\eps_i(x,t) = 1 + \eps \psi(x) \sin(2 \pi t/\tau) + o(\eps)$ as $\eps \to 0$, where $\tau$ is an ``intrinsic'' period that depends on $V_i, \Omega, \sigma$ and not on $T$. Finally, some numerical analysis is performed to show that at least in some examples, a ``periodic'' structure in time survives beyond the perturbative regime $\eps \to 0$.

If $V_i \in C^1$, a simple Taylor expansion shows that  \eqref{uniquecond} holds if the symmetrization of the Jacobian matrix $JV(\cdot)$ is everywhere positive semi-definite. Here, we will suppose instead that
\begin{equation}\label{Vass}
\begin{split}
\bullet &\, \text{$V_i \in C^\infty((0,+\infty) \times (0,+\infty))$ are bounded on $\R^2$}, \\
\bullet &\, JV := JV(1,1) = 
\begin{bmatrix}
 \partial_{m_1} V_1(1, 1) &  \partial_{m_2} V_1(1, 1) \\
  \partial_{m_1} V_2(1, 1) &  \partial_{m_2} V_2(1, 1)
\end{bmatrix} \\
& \text{has real eigenvalues $(a_1, a_2)$ with opposite signs, i.e.  $a_1 < 0 \le a_2$.}
\end{split}
\end{equation}
Moreover, we will not make any assumption on the smallness of $T$.

We will focus on two different models that fall into \eqref{Vass}. To describe them (and our assumption  \eqref{Vass} itself), let us consider the simple situation
\begin{equation}\label{Vex}
\begin{cases}
V_1(m_1, m_2) = \alpha_1 g(m_1) + \gamma g(m_2), \\
V_2(m_1, m_2) = \gamma g(m_1) + \alpha_2 g(m_2),
\end{cases}
\end{equation}
where $\alpha_1, \alpha_2, \gamma \in \R$ and $g(\cdot)$ is a smooth, bounded and increasing function such that $g'(1) = 1$. Then,
\[
JV = \begin{bmatrix}
\alpha_1 &  \gamma \\
\gamma &  \alpha_2
\end{bmatrix},
\]
and \eqref{Vass} basically holds whenever $\alpha_1 \alpha_2 \le \gamma^2$. Since the coefficients $\alpha_i$ determine self-interaction costs within each population, and $\gamma$ is related to the cost of interaction between players of different populations, \eqref{Vass} somehow requires that the latter cost has a stronger effect than the former. Anyhow, \eqref{Vass} also includes a ``decoupled'' case, that occurs in the example \eqref{Vex} when $\gamma = 0$, namely when the two populations do not interact within each other. Here, one should impose $\alpha_1 < 0$ and $\alpha_2 \ge 0$. This means that the first population aims at aggregation (the cost of a typical player is decreasing with respect to the density $m_1$), while players of the other one prefer to avoid crowded regions. System \eqref{MFG0} then consists of two decoupled MFG systems for $(u_1, m_1)$ and $(u_2, m_2)$; the latter enjoys uniqueness of solutions, while the former is so called ``non-monotone'' or ``focusing''. The ``non-monotone'' setting has been recently considered in \cite{CCes, cir16, CTon, Go16}, and exhibits a much more complicated behaviour than the monotone one (existence of several solutions, concentration, ...). Since the study of the non-monotone case is still at an early stage, we believe that an analysis of the decoupled system can be of interest on its own. We will consider this single-population, non-monotone problem in Section \ref{spa}.

We will also focus on a truly two-population model, inspired by the work of T. Schelling on urban settlements \cite{ABC}, and somehow resembling \eqref{Vex} with $\alpha_1 \alpha_2 \le \gamma^2$ and $\alpha_i, \gamma > 0$, that is when the two populations tend to avoid each other. In \cite{ABC}, some numerical experiments indicated the possible existence of solutions that are unstable in time, namely oscillating between different states; that numerical insight has been one of the main motivations of the present work, and is carried over here in Section \ref{spb}. 

Let us discuss the choice of the initial-final data \eqref{ifcond}, that will be fixed throughout the paper. Assuming without loss of generality that $|\Omega|=1$, we have \[
m_i(\cdot,0) \equiv 1 \qquad \text{on $\Omega$}.
\]
Then, for all $T > 0$, \eqref{MFG0} has the {\it trivial} solution $(\bar u, \bar m) = ((T-t)V_1(1,1), (T-t)V_2(1,1), 1, 1)$. In other words, the constant state is a stable in time equilibrium for every choice of $T$. Here, we ask whether or not there exist equilibria where the populations move away from the constant state as $T$ varies. More precisely, denoting by $S$ the closure of the set of non-trivial (classical) solutions to \eqref{MFG0}, namely
\[
S := {\rm cl}\{(u, m, T) \in [C^{4+\alpha, 2+\alpha/2}({Q_T})]^4 \times (0,\infty) \, : \, \text{$(u, m)$ is a solution to \eqref{MFG0} and  $(u, m) \neq (\bar u, \bar m)$}\},
\]
we aim at giving a description of the set $S$ (here, $\alpha \in (0, 1)$ is fixed, and $C^{4+\alpha, 2+\alpha/2}(Q_T)$ denotes the standard H\"older parabolic space, see Section \ref{s_np}). Since for all $T$, \eqref{MFG0} has an explicit solution which is the trivial one, it is natural to treat $T$ as a parameter, and to fit the problem into a bifurcation framework, namely to search for (connected) subsets of $S$ containing $(\bar u, \bar m, T^*)$ for some $T^*$; such a $T^*$ will be called a bifurcation point. Based on topological methods, we will implement here the classical tools of local and global bifurcation (see, e.g., \cite{Kiel}). To do so, one has first to consider solutions of \eqref{MFG0} as fixed points of a non-linear operator $G(x, m)$. Then, for $T^*$ to be a bifurcation point, it is necessary that the linearized operator around the trivial solution has a non-trivial kernel. This property becomes sufficient if some additional ``transversality'' condition involving derivatives of $G$ holds. The study of the linearized operator will be carried out analyzing expansions of the variables that are based on eigenfunctions of the (Neumann) Laplacian. Global bifurcation methods will then lead to the existence of continua $C \subset S$, or {\it branches}, of non-trivial solutions emanating from critical bifurcation times $T^*$. Finally, local bifurcation results will give an explicit parametrization of $C_n$ close to $T^*$, that will provide a rather precise qualitative description of solutions. We point out that bifurcation techniques have been used extensively to study several nonlinear scalar PDEs and systems of PDEs, but their implementation in MFG is rather new; as far as we know, they have been used only in \cite{CV} to study stationary problems.

Before stating the main result of this paper, let us denote by  $(\lambda_k)_{k \ge 0}$ the non-decreasing sequence  of eigenvalues
of $-\Delta$ with homogeneous Neumann boundary conditions, namely $\lambda_k$ be such that
\begin{equation}\label{eigen_sys}
\begin{cases}
- \Delta \psi_k = \lambda_k \psi_k & \text{in $\Omega$}, \\
\partial_n \psi_k = 0 & \text{on $\partial \Omega$},
\end{cases}
\end{equation}
for some eigenvector $\psi_k \in C^{\infty}(\overline{\Omega})$; let $(\psi_k)_{k \ge 0}$ be renormalized such that it constitutes an orthonormal basis of $L^2(\Omega)$. Note that the first eigenvalue $\lambda_0$ is zero, with associated constant eigenfunction $\psi_0$. 

Denote also by $\Xi$ a $2\times2$ invertible square matrix whose columns are the eigenvectors associated to $a_1$ and $a_2$, so that 
\begin{equation}\label{diagonal}
\begin{bmatrix}
a_1 &  0 \\
  0 &  a_2
\end{bmatrix}
= \Xi \cdot JV \cdot \Xi^{-1}.
\end{equation}

We have the following existence result.

\begin{thm}\label{glob_bifo_mfg} Suppose that \eqref{Vass} holds, that $\lambda_1$ is a simple eigenvalue of \eqref{eigen_sys}, $\sigma^2 \lambda_1 < -a_1 < \sigma^2 \lambda_2$, and that $T^*_n > 0$ satisfies
\begin{equation}\label{TC}
T^*_n = \frac{1}{\sqrt{\lambda_{1}(-a_1 - \sigma^2 \lambda_{1})}}\left[n\pi - \arctan\left(\frac{1}{\sigma}\sqrt{\frac{-a_1- \sigma^2 \lambda_{1}}{\lambda_{1}}}\right)\right]
\end{equation}
for some $n \in \mathbb{N}$. Then, $(\bar u, \bar m, T^*_n) \in S$. Let $C_n$ be the connected component of $S$ to which $(\bar u, \bar m,T^*_n)$ belongs. Then
\begin{itemize}
\item[{\it i)}] $C_n$ contains some $(\bar u, \bar m, \widetilde T)$, where $T^*_n \neq \widetilde T$, or
\item[{\it ii)}] $C_n$ is unbounded.
\end{itemize}

Finally, $C_n$ is a continuously differentiable curve in a neighbourhood of $T^*_n$, parametrized by
\begin{equation}\label{expa}
T = T^*_n + O(\eps), \quad (m_1(x,s), m_2(x,s)) = (1,1) + \eps \psi_{1}(x) \Xi \cdot \left[\sin \left(\sqrt{\lambda_{1}(-a_1 - \sigma^2 \lambda_{1})}\, s\right) , 0\right]^T + o(\eps).
\end{equation}
\end{thm}

 The theorem states the existence of continua $C_n$, or {\it branches}, that emanate from the trivial solution. Such continua can be either unbounded in $C^{4+\alpha, 2+\alpha/2} \times (0, + \infty)$ (that is, there are sequences $(u_j,m_j,T_j) \subset C_n$ such that $T_j + \|(u_j,m_j)\|_{[C^{4+\alpha, 2+\alpha/2}(Q_T)]^4} \to \infty$ as $j \to \infty$), or collapse back to another bifurcation point. See, e.g., Figure \ref{bif_ex}. Though in our numerical experiments the second possibility does not seem to occur (see Figures \ref{fig_2branches}, \ref{fig_5branches}, \ref{fig_sch_br}), namely all branches appear to be unbounded, we are not able at this stage to say whether or not this is a general fact.

Let us comment on the local parametrization \eqref{expa}, that describes non trivial solutions $m$ close to bifurcation points: $m$ can be seen as a perturbation of the trivial state $(1, 1)$, with split space-time dependance, of the form $\psi_{1}(x) \sin(2 \pi s/\tau)$. We observe that the period $\tau = 2\pi (\lambda_{1}(-a_1 - \sigma^2 \lambda_{1}))^{-1/2}$ {\it does not depend on $n$}, and $T^*_n$ reads
\[
T^*_n = \frac{\tau}{2}(n - \delta),
\]
for some $\delta = \delta(\sigma, \Omega, a_1)  < 1/2$. This means that, for $T$ close to $T^*_n$, \eqref{MFG0} has solutions such that $m$,  starting at $s = 0$ from the constant state, switches $n-1$ times between (approximately) the profiles $1 + \eps \psi_1(x)$ and $1 -\eps \psi_1(x)$ as $s$ increases. This {\it instability} between different states will be evident in numerical simulations that will be presented in the second part of this work. We observe that this oscillatory behaviour is quite in contrast with the stability of the ``monotone'' case, where the unique equilibrium $m$ is expected to converge to a unique stable state as $T \to \infty$, see \cite{CLLP}. Note that $T^*_n \to \infty$ as $n \to \infty$, therefore one can find solutions $m^\eps_i$ on $(0, T^\eps) \approx (0, T^*_n)$ exhibiting an arbitrarily large number of small oscillations.

Some numerical analysis will be carried out not only to visualize the shape of solutions predicted by Theorem \ref{glob_bifo_mfg} in particular models, but also to have a clue of their qualitative behaviour when $T$ is not close to bifurcation points, that is, far from $T^*_n$. Some properties will be pointed out, e.g. the conservation of the number of oscillations (in time) among every branch, but we stress that rigorous proofs of such observations are not available at this stage. Still, numerical insights can be matter of future work.

Theorem \ref{glob_bifo_mfg}, and in particular the representation formula \eqref{expa} in the perturbative regime $\eps \to 0$ (i.e. $T \to T^*_n$), suggest that there might be truly periodic in time solutions to the MFG system \eqref{MFG0} defined over $(-\infty, \infty)$. This is the matter of a subsequent work \cite{CNur}, that uses several ideas proposed in the present paper. We mention that when the time horizon is the whole line $(-\infty, \infty)$, solutions to the linearized system consist generally of a vector space with even dimension, that prevent the use of bifurcation results involved here. This issue is circumvented using the special variational (Hamiltonian) structure of \eqref{MFG0} in particular cases. Note that periodic solutions can be also obtained from travelling wave solutions if the state space $\Omega$ with Neumann conditions is replaced by a space-periodic environment (i.e. the flat torus). In this direction, results on so-called congestion problems appeared in \cite{GS}, see also \cite{YMMS}.

We finally note that \eqref{MFG0} is a MFG system with quadratic Hamiltonians, so it has a very special structure (see \cite{LL07, CardaNotes}). Actually, the existence statement of Theorem \ref{glob_bifo_mfg} holds (under suitable assumptions) if the Hamiltonians have a more general form $H(x, p)$; the key point is that, to perform a bifurcation analysis, one needs this special structure only locally, that is, close to the trivial state $\bar u$. In particular, it is basically sufficient to have $\nabla H(\cdot, 0) = 0$ and $D^2 H(\cdot, 0) = \kappa I$ on $\Omega$ for some $\kappa > 0$. This possible generalization will be discussed in Remark \ref{addrem2}, among other considerations on further relaxations of the standing assumptions.

\smallskip

The paper is organized as follows. Section \ref{s_bif} is devoted to the bifurcation analysis for \eqref{MFG0}, and the proof of Theorem \ref{glob_bifo_mfg}. In Section \ref{s_appl}, some applications to two different models will be presented, with a discussion on several observations based on numerical evidences.

\medskip

{\bf Acknowledgements.} This work has been partially supported by the Fondazione CaRiPaRo Project ``Nonlinear Partial Differential Equations: Asymptotic Problems and Mean-Field Games''.

\section{Bifurcation analysis}\label{s_bif}

\subsection{Notations and preliminaries}\label{s_np} Let $Q_T := \Omega \times (0,T)$, and for all $k \in \mathbb N$, $0< \beta < 1$, let $C^{2k+\beta,k+\beta/2}(Q_T)$ be set of continuous functions on $Q_T$ having derivatives $\partial^r_tD^{\omega}_xu$ that are $\beta$-H\"older continuous in the $x$-variable and $\beta/2$-H\"older continuous in $t$-variable for $2r+|\omega| \le 2k$ (see, e.g., \cite{Lieber}). For brevity, $Q := Q_1$.

We will denote by $\nu : \partial \Omega \to \RsetN$ the outer normal vector field at $\partial \Omega$. For a linear operator $\Lc$, $N(\Lc)$ and $R(\Lc)$ will be its kernel and its image respectively. Finally, $(\lambda_k, \psi_k$) will be the eigenpairs of \eqref{eigen_sys}.

We will use the following version of the global Rabinowitz Bifurcation Theorem. Let $X$ be a Banach space, and $\Fc \in C(X \times \R, X)$. Denote by $\mathcal S$ the closure of the set of nontrivial solutions of
\[
\Fc(x,\lambda) = 0
\]
in $X \times \R$.

\begin{thm}\label{global_bif} Assume that $\Fc(x,\lambda) = x- \Gc(x, \lambda)$, where $\Gc : X \times \R \to X$ is a compact mapping \footnote{that is, $\Gc$ is continuous, and the image under $\Gc$ of any bounded subsubset of $X \times \R$ has compact closure in $X$.}, $D_x \Fc(0,\cdot) = I - D_x \Gc(0,\cdot) \in C(\R, L(X,X))$ and it is differentiable. Suppose that zero is a geometrically simple (isolated) eigenvalue of $D_x \Fc(0,\lambda_0)$, and
\begin{equation}\label{transversality}
D^2_{x\lambda} \Fc(0,\lambda_0)[v_0] \notin R(D_x \Fc(0,\lambda_0)),
\end{equation}
where $v_0$ is the element spanning $N(D_x \Fc(0,\lambda_0))$. Then, $(0,\lambda_0) \in \mathcal S$. 

Let $\mathcal C$ be the connected component of $\mathcal S$ to which $(0,\lambda_0)$ belongs. Then
\begin{itemize}
\item[{\it i)}] $\mathcal C$ is unbounded, or
\item[{\it ii)}] $\mathcal C$ contains some $(0,\lambda_1)$, where $\lambda_0 \neq \lambda_1$.
\end{itemize}
\end{thm}

For a proof of Theorem \ref{global_bif} we refer to \cite{Kiel}, Theorem II.3.3, and the discussion after Theorem II.4.4 (in particular, p. 213).

\begin{rem}\label{param_loc} If $\Fc \in C^2(X \times \R, X)$, $\mathcal C$ in Theorem \ref{global_bif} is a continuously differentiable curve in a neighbourhood of $(0,\lambda_0)$, parametrized by
\[s \mapsto (x, \lambda) = (s v_0 + s \psi (s), \lambda_0 + \varphi(s)),\]
where $\varphi(0)= 0, \psi(0)=0$, see \cite[Theorem I.5.1, Corollary I.5.2]{Kiel}.
\end{rem}

\smallskip
Throughout this section, we will set $\sigma = 1$ for simplicity, the analysis for different values of $\sigma > 0$ being identical (see Remark \ref{mfg_loc}). 

To treat \eqref{MFG0} with bifurcation methods, we shall do first a change of variables that involves a time rescaling. In particular let, on $\Omega \times [0,1]$,
\[
\begin{cases}
\su_i(x,t) := u_i(x,T t) - \bar u_i(x,T t) = u_i(x,T t) + T(t-1)V_i(1,1),  \\
\sm_i(x,t) := m_i(x,T t) - \bar m_i(x,T t) = m_i(x,T t) -1.
\end{cases}
\]
Then, \eqref{MFG0} becomes
\begin{equation}\label{MFG1}
\begin{cases}
-\frac{1}{T}\partial_t  \su_i - \sigma \Delta \su_i  + \frac{1}{2}|\nabla \su_i |^2 = V_i(1 + \sm_1, 1 + \sm_2) - V_i(1, 1), & \text{in $Q, \, i = 1,2$,} \\
\frac{1}{T}\partial_t \sm_i -  \sigma \Delta \sm_i  -  \Delta \su_i   - \diverg(\nabla \su_i \, \sm_i ) = 0
\end{cases}
\end{equation}
with boundary conditions
\[
\begin{cases}
\partial_\nu \su_i = \partial_\nu \sm_i = 0 &\text{on $\partial \Omega \times (0,1)$} , \\
\sm_i(x,0)= \su_i(x,1)= 0 & \text{on $\Omega$}.
\end{cases}
\]
Note that the space-time domain is fixed in \eqref{MFG1}, and $T$ appears as a parameter in the equations. Moreover, $(\su, \sm) \equiv (0, 0, 0, 0)$ is the trivial solution for all $T$. 

We will consider solutions to \eqref{MFG1} as fixed points of a non-linear functional. For any fixed $\alpha \in (0,1)$, let
\[
X := \{(\sm_1, \sm_2) \in C^{4+\alpha, 2+\alpha/2}(Q) \times C^{4+\alpha, 2+\alpha/2}(Q) :  \text{$\partial_\nu \sm_i = 0$ on $\partial \Omega \times (0,1)$ and $\sm_i(\cdot,0)=0$ on $\Omega$} \}.
\]
For any $\sm \in X, T > 0$, define the mapping $\mu = \Gc(\sm, T)$ in such a way that $\mu$ is the classical solution of
\begin{equation}\label{Gdef}
\begin{cases}
-\frac{1}{T} \partial_t \su_i - \Delta \su_i + \frac{1}{2}|\nabla \su_i|^2 = V_i(1 + \sm_1, 1 + \sm_2) - V_i(1, 1),& \text{in $Q$,} \\
\frac{1}{T} \partial_t \mu_i -  \Delta \mu_i -  \Delta \su_i  - \diverg(\nabla \su_i\, \mu_i) = 0, \\
\partial_\nu \su = \partial_\nu \mu = 0 &\text{on $\partial \Omega \times (0,1)$} , \\
\mu(x,0)= \su(x,1)= 0 & \text{on $\Omega$}.
\end{cases}
\end{equation}

It is clear that solutions to \eqref{MFG1} (and therefore to \eqref{MFG0}) are fixed points of $\Gc$, or, in other words, zeroes of the following functional
\[
\Fc(\sm, T) := -\sm + \Gc(\sm, T) \quad \forall (\sm, T) \in X \times (0, \infty).
\]
Note that $\Fc(0, T) = 0$ for all $T > 0$, so it is natural to treat $T$ as a bifurcation parameter.

With this point of view in mind, we aim at applying Theorem \ref{global_bif}. We begin by showing that, for small $T$, \eqref{MFG1} has only the trivial solution, namely $\Gc$ has just a trivial fixed point.

\begin{lem}\label{unique} There exists $\overline T > 0$ such that if \eqref{MFG1} has a (classical) non-trivial solution $(\su, \sm, T)$, then $T > \overline T$.
\end{lem}

\begin{proof} By \cite[Corollary 3.1]{BCir}, there exists $\overline T > 0$ such that \eqref{MFG0}, and therefore \eqref{Gdef}, has a unique classical solution for all $T \le \overline T$. Since  \eqref{MFG0} has the trivial solution for all $T > 0$, a non-trivial solution may exist only if $T > \overline T$. \end{proof}

We now show that $\Fc$ is a so-called {\it compact perturbation of the identity}.

\begin{lem}\label{compcont} $\Gc : X \times (0, \infty) \to X$ is a compact mapping, and $D_\sm \Gc(0, \cdot) \in C(\R, L(X,X))$.
\end{lem}

\begin{proof}
We claim that, by standard parabolic regularity, $\Gc(\sm, T) \in C^{6+\alpha, 3+\alpha/2}(Q) \times C^{6+\alpha, 3+\alpha/2}(Q)$ for all $(\sm, T) \in X \times (0, \infty)$. In particular, any bounded set of the form $Y \times (t_0, t_1) \subset X \times (0, \infty)$, $t_0 > 0$, is mapped by $\Gc$ into a bounded subset of $C^{6+\alpha, 3+\alpha/2}(Q) \times C^{6+\alpha, 3+\alpha/2}(Q)$, whose closure is compact in $X$. Indeed, if $(\sm, T) \in Y \times (t_0, t_1)$, $\nabla \su_i$ are bounded in some H\"older space by \cite[Theorem XIII.13.15]{Lieber} and  \cite[Theorem XIII.13.16]{Lieber}. Parabolic regularity for linear equations in divergence and non-divergence form (see, e.g., \cite[Theorem IV.4.30]{Lieber} and \cite[Theorem IV.4.31]{Lieber}) then yields the claim. Continuity of $\Gc$ follows by stability of the Hamilton-Jacobi-Bellman and Fokker-Planck equations in \eqref{Gdef}, which is a standard by-product of parabolic regularity and respective uniqueness of solutions.

By computation (and again stability of linear parabolic equations with respect to their coefficients), $\Gc$ is G\^ateaux differentiable at every $(\sm, T) \in X \times (0, \infty)$, and $z = D_\sm \Gc(\sm, T)[h]$ satisfies
\[
\begin{cases}
-\frac{1}{T}\partial_t v_i - \Delta v_i + \nabla \su_i \cdot \nabla v_i = \partial_{m_1} V_i(1 + \sm_1, 1+ \sm_2)h_1 + \partial_{m_2} V_i(1 + \sm_1, 1+ \sm_2)h_2,& \text{in $Q$,} \\
\frac{1}{T}\partial_t z_i -  \Delta z_i -  \Delta v_i  - \diverg(\nabla v_i \, \mu_i)  - \diverg(\nabla \su_i\, z_i) = 0, \\
\partial_\nu v_i = \partial_\nu z_i = 0 &\text{on $\partial \Omega \times (0,1)$} , \\
z_i(x,0)= v_i(x,1)= 0 & \text{on $\Omega$}.
\end{cases}
\]
By continuity of $\sm \mapsto (\su, \mu)$ (uniform w.r.t $\|h\|_{C^{4+\alpha, 2+\alpha/2}(\Omega) \times C^{4+\alpha, 2+\alpha/2}(\Omega)} = 1$) and stability of the equations , $\Gc$ is also Fr\'echet differentiable. Note finally that $z = D_\sm \Gc(0, T)[h]$ satisfies
\begin{equation}\label{diff_zero}
\begin{cases}
- \partial_t v_i - T \Delta v_i = T[\partial_{m_1} V_i(1, 1) h_1 + \partial_{m_2} V_i(1, 1)h_2],& \text{in $Q$,} \\
\partial_t z_i -  T \Delta z_i - T \Delta v_i  = 0, \\
\partial_\nu v_i = \partial_\nu z_i = 0 &\text{on $\partial \Omega \times (0,1)$} , \\
z_i(x,0)= v_i(x,1)= 0 & \text{on $\Omega$}.
\end{cases}
\end{equation}
and continuity of the differential $D_\sm \Gc$ can be easily verified.
\end{proof}

\subsection{The linearized system}

The application of Theorem \ref{global_bif} relies on the study of the linearization of \eqref{MFG1}. We will in particular analyse the linearized system \eqref{diff_zero} using expansions over the eigenfunctions $\psi_k(x)$ of the Neumann $-\Delta$ (as in \eqref{eigen_sys}): every $\sm_i$ admits a unique representation in terms of time-dependent coefficients $\sm_{i,k}(t) := \int_\Omega \sm_i(x,t) \psi_k(x)dx$, that is,
\begin{multline}
X = \{(\sm_1, \sm_2) \in C^{4+\alpha, 2+\alpha/2}(Q) \times C^{4+\alpha, 2+\alpha/2}(Q): \  \sm_i(x,t) = \sum_{k \ge 0} \sm_{i,k}(t) \psi_k(x), \\ \text{ for some $(\sm_{i,k})_{k \ge 0} \subset C^{2+\alpha/2}([0,1])$ s.t. $\sm_{i,k}(0) = 0\, \forall i=1,2, k\ge 0$}\}.
\end{multline}

In the sequel, we will identify an $f \in C^{4+\alpha, 2+\alpha/2}(Q)$ with its associated sequence of coefficients $(f_k)_{k \ge 0} \subset C^{2+\alpha/2}([0,1])$.

Note that, setting $z = \Xi \cdot\bar  z$, $v = \Xi \cdot\bar  v$ and $h = \Xi \cdot\bar  h$ (in general, we will use the notation $x = \Xi\cdot \bar x$ for any vector in $x \in \R^2$, treating $x, \bar x$ as column vectors), then \eqref{diff_zero} decouples, and becomes equivalent to
\begin{equation}\label{diff_zero_xi}
\begin{cases}
- \partial_t \bar v_i - T \Delta \bar v_i = T a_i \bar h_i,& \text{in $Q$,} \\
\partial_t \bar z_i -  T \Delta \bar z_i - T \Delta \bar v_i  = 0, \\
\partial_\nu v_i = \partial_\nu z_i = 0 &\text{on $\partial \Omega \times (0,1)$} , \\
z_i(x,0)= v_i(x,1)= 0 & \text{on $\Omega$},
\end{cases}
\end{equation}
where $a_i$, $\Xi$ are as in \eqref{diagonal}. Let
\[
\Lc(T) := D_\sm \Fc(0, T) =  I - D_\sm \Gc(0, T).
\]

\begin{lem}\label{L_coeff} For any $z, \tilde h \in X$, we have that $z = \Lc(T)[\tilde h]$ if and only if, setting $h = \tilde h - z$, $h_{i,k}$ satisfies
\begin{equation}\label{L_coeff_s}
\begin{cases}
h_{i,k}''(t) = T^2 \lambda_k^2 h_{i,k}(t) + T^2  \lambda_k [\partial_{m_1} V_i(1, 1) (h_{1,k}(t) + z_{1,k}(t))+ \partial_{m_2} V_i(1, 1) (h_{2,k}(t) + z_{2,k}(t))], \\
h_{i,k}(0)=0, \quad h'_{i,k}(1) + T \lambda_k h_{i,k}(1)  = 0
\end{cases}
\end{equation}
for all $k \ge 0$, $i=1,2$.
\end{lem}

Note that, in a slightly more compact form, \eqref{L_coeff_s} reads
\[
\begin{cases}
h_k''(t) = T^2 \lambda_k^2 h_k(t) + T^2  \lambda_k \, JV \cdot (h_{k}(t) + z_{k}(t)) & \text{in $(0,1)$,} \\
h_k(0)=0, \quad h'_k(1) + T \lambda_k h_k(1)  = 0
\end{cases}
\]

\begin{proof}  Note that $h$ is such that $D_\sm \Gc(0, T)[h+z] = h$, $h_i(\cdot, 0) = 0$ on $\Omega$ and $\partial_\nu h_i = 0$ on $\partial \Omega \times (0,1)$. Recalling that $(-\Delta u)_k = u_k \lambda_k$, if $g_{i,k} := [\partial_{m_1} V_i(1, 1) (h_{1,k} + z_{1,k}) + \partial_{m_2} V_i(1, 1) (h_{2,k} + z_{2,k})]$, then projecting \eqref{diff_zero} onto $\psi_k$ yields
\[
\begin{cases}
- v_{i,k}' + T \lambda_k v_{i,k} =T g_{i,k},& \text{in $(0, 1)$,} \\
h_{i,k}' + T \lambda_k h_{i,k} + T \lambda_k v_{i,k}  = 0 \\
h_{i,k}(0) = v_{i,k}(1) = 0
\end{cases}
\]
for all $k \ge 0$, $i=1,2$. By an easy computation, this system is equivalent to a second order equation, namely
\[
h_{i,k}'' + T \lambda_k h_{i,k}' = - T \lambda_k v_{i,k}' = T \lambda_k(h'_{i,k} + T \lambda_k h_{i,k} + Tg_{i,k}).
\]
Moreover, $h_{i,k}(0) = 0$, and $0 = h_{i,k}'(1) + T \lambda_k h_{i,k}(1) + T \lambda_k v_{i,k}(1) = h_{i,k}'(1) + T \lambda_k h_{i,k}(1)$,
and the statement follows.
\end{proof}

\begin{lem}\label{whoisLp} We have that $w = \Lc'(T)[\tilde h] := D_{\sm T} \Fc(0, T)[\tilde h] $ if and only if $w$ satisfies
\[
\begin{cases}
-w_k''(t) = 2T\lambda_k^2 \tilde{h}_k + 2T\lambda_k JV \cdot \tilde{h}_k -2T \lambda_k^2 z_k - T^2 \lambda_k^2 w_k & \text{in $(0,1)$,} \\
w_k(0)=0, \quad w'_k(1) + T \lambda_k w_k(1)  = \lambda_k({\tilde h}_k(1) - z_k(1))
\end{cases}
\]
for all $k \ge 0$, where $z = \Lc(T)[\tilde h]$.
\end{lem}

\begin{proof} Follows by differentiating \eqref{L_coeff_s} (in its compact form). \end{proof}

\begin{lem}\label{whoisker} The kernel $N(\Lc(T))$ is spanned by $h^0, h^1, \ldots$, where
\[
h^j(x,t) = \psi_{k_j}(x) \Xi \cdot \left[\sin \left(T\sqrt{\lambda_{k_j}(-a_1 - \lambda_{k_j})}\, t\right) , 0\right]^T \quad \text{on $Q$}
\]
are defined by the eigenvalues $0 < \lambda_{k_j} < a_1$ that satisfy the equation
\begin{equation}\label{kernel_eq}
\tan \left(T\sqrt{\lambda_{k_j}(-a_1 - \lambda_{k_j})}\right) = -\sqrt{\frac{-a_1- \lambda_{k_j}}{\lambda_{k_j}}}.
\end{equation}
\end{lem}
Note that the set of $\lambda_{k_j}$ satisfying \eqref{kernel_eq} is finite, since $\lambda_k \to \infty$ as $k \to \infty$. 

\begin{proof} By Lemma \ref{L_coeff}, $\Lc(T)[h] = 0$ if and only if, for all $k \ge 0$, $i=1,2$,
\begin{equation}\label{sys9}
\begin{cases}
h_{i,k}''(t) = T^2 \lambda_k^2 h_{i,k}(t) + T^2  \lambda_k [\partial_{m_1} V_i(1, 1) h_{1,k}(t) + \partial_{m_2} V_i(1, 1) h_{2,k}(t) ] & \text{in $(0,1)$,} \\
h_{i,k}(0)=0, \quad h'_{i,k}(1) + T \lambda_k h_{i,k}(1)  = 0.
\end{cases}
\end{equation}
We use the linear transformation on $\R^2$ induced by $\Xi$ to decouple the system: the coefficients $\bar h_k$ associated to $\bar h = \Xi^{-1} h$ satisfy
\begin{equation}\label{sys95}
\begin{cases}
\bar h_{i,k}''(t) = T^2 \lambda_k (\lambda_k +a_i) \bar h_{i,k}(t) & \text{in $(0,1)$,} \\
\bar h_{i,k}(0)=0, \quad \bar h'_{i,k}(1) + T \lambda_k \bar h_{i,k}(1)  = 0.
\end{cases}
\end{equation}
For any $i,k$, this second order equation does not have non-trivial solutions if $\lambda_k + a_i \ge 0$, so $\bar h_{2,k} \equiv 0$ for all $k$. On the other hand, it is possible to check that $\bar h_{1,k_{j}}(t) = A \sin(\omega \, t) + B\cos(\omega \, t)$ is a non-trivial solution if and only if $\omega = T\sqrt{\lambda_{k_j}(-a_1 - \lambda_{k_j})}$, $B = 0$ and \eqref{kernel_eq} holds.
\end{proof}

\begin{lem}\label{rankL} $z\in R(\Lc(T))$ if and only if
\[
\int_0^1 [\Xi^{-1}_{11} z_{1,k_j}(t) + \Xi^{-1}_{12}z_{2,k_j}(t)] \sin \left(T\sqrt{\lambda_{k_j}(-a_1 - \lambda_{k_j})}\, t\right) dt = 0
\]
for all $\lambda_{k_j}$ satisfying \eqref{kernel_eq}.
\end{lem}

\begin{proof} Denote by $K = K_{a,k,T}$ the inverse of the one-dimensional Laplacian on $(0,1)$ with Robin conditions, i.e. $u = K[f]$ if and only if (in the weak sense)
\[
u''(t) = T^2 (\lambda_k + a_1)  \lambda_k f(t) \quad \text{in $(0,1)$,} \qquad u(0)=0, \quad u'(1) + T \lambda_k u(1)  = 0.
\]
Note that this operator is bounded, linear, self-adjoint and compact on $L^2((0,1))$.

 By Lemma \ref{L_coeff}, $z \in R(\Lc(T))$ if and only if there exists $h $ that satisfies
\[
\begin{cases}
h_k''(t) = T^2 \lambda^2_k h_k(t) + T^2 \lambda_k \, JV \cdot (h_k(t) + z_k(t)) & \text{in $(0,1)$,} \\
h_k(0)=0, \quad h'_k(1) + T \lambda_k h_k(1)  = 0
\end{cases}
\]
for all $k$. Passing to the transformed variables $\bar h$, $\bar z$, the system reads
\[
\begin{cases}
\bar h_{i,k}''(t) = T^2 (\lambda_k + a_i) \lambda_k \bar h_{i,k}(t) + T^2 \lambda_k a_i \bar z_{i,k}(t) & \text{in $(0,1)$,} \\
\bar h_{i,k}(0)=0, \quad \bar h'_{i,k}(1) + T \lambda_k \bar h_{i,k}(1)  = 0
\end{cases}
\]
The set of equations with $i=2, k \ge 0$ has solutions $\bar h_{2,k}$, since $\lambda_k + a_2 \ge 0$. For $i=1$ and any fixed $k$, the problem can be restated in finding $\bar h_{1,k}$ so that
\[
\bar h_{1,k} =  K[\bar h_{1,k}] + \frac{a_1}{(\lambda_k + a_1)} K[\bar z_{1,k}],
\]
that is possible if and only if $K[\bar z_{1,k}] \in R(I - K)$. By the Fredholm alternative, $R(I - K) = N(I - K)^\perp$. One then verifies that $N(I - K)$ is made up of solutions to \eqref{sys95}, so $K[\bar z_{1,k}] \in R(I - K)$ if and only if
\begin{multline*}
0 = \int_0^1 K[\bar z_{1,k}](t) \sin \left(T\sqrt{\lambda_{k_j}(-a_1 - \lambda_{k_j})}\, t\right) dt = \\ \int_0^1 \bar z_{1,k}(t) K\left[\sin \left(T\sqrt{\lambda_{k_j}(-a_1 - \lambda_{k_j})}\, \cdot\right)\right](t) dt = \int_0^1  \bar z_{1,k}(t) \sin \left(T\sqrt{\lambda_{k_j}(-a_1 - \lambda_{k_j})}\, t\right) dt,
\end{multline*}
whenever $\lambda_{k_j}$ satisfies \eqref{kernel_eq} (otherwise $K[\bar z_{1,k}] \in R$ automatically), that concludes the proof.

\end{proof}

\subsection{Global bifurcation}

We are ready to prove the main result of this section, concerning global bifurcation for \eqref{MFG1}. The proof of Theorem \ref{glob_bifo_mfg} will directly follow by going back to the unknowns $u,m$.

\begin{thm}\label{glob_bifo} Suppose that $\lambda_1$ is a simple eigenvalue, $\lambda_1 < -a_1 < \lambda_2$, and that $T^* > 0$ satisfies
\begin{equation}\label{Tcond}
\tan \left(T^*\sqrt{\lambda_{1}(-a_1 - \lambda_{1})}\right) = -\sqrt{\frac{-a_1- \lambda_{1}}{\lambda_{1}}}.
\end{equation}

Then, $(0,T^*) \in \mathcal S$. Let $\mathcal C$ be the connected component of $\mathcal S$ to which $(0,T^*)$ belongs. Then
\begin{itemize}
\item[{\it i)}] $\mathcal C$ contains some $(0,T_1)$, where $T^* \neq T_1$, or
\item[{\it ii)}] $\mathcal C$ is unbounded.
\end{itemize}
\end{thm}

\begin{rem} If $T$ does not satisfy \eqref{Tcond}, the linearized system associated to $(\bar u, \bar m)$ does not have non-trivial solutions; using the jargon introduced in \cite{BriCar}, $(\bar u, \bar m)$ is then a {\it stable} solution, so it is in particular isolated.
\end{rem}

\begin{proof} We will apply Theorem \ref{global_bif}. 
We have to check that zero is a geometrically simple (isolated) eigenvalue of $D_m \Fc(0,T^*)$ and the transversality condition \eqref{transversality}. The first assertion comes from Lemma \ref{whoisker}, since $\lambda_1$ is a simple eigenvalue and $\lambda_1 < -a_1 < \lambda_2$. Moreover, $N(D_m \Fc(0,T^*))$ is spanned by
\[
h^0(x,t) = \psi_{1}(x) \Xi \cdot \left[\sin \left(T\sqrt{\lambda_{1}(-a_1 - \lambda_{1})}\, t\right) , 0\right]^T.
\]

Since $D_m \Fc(0,T^*)[h^0] = 0$, by Lemma \ref{whoisLp} we have that $w = \Lc'(T^*)[h^0]$ if and only if 
\begin{equation}\label{sys11}
\begin{cases}
-w_k''(t) = 2T^*\lambda_k^2 h^0_k + 2T^*\lambda_k JV \cdot h^0_k - (T^*)^2 \lambda_k^2 w_k & \text{in $(0,1)$,} \\
w_k(0)=0, \quad w'_k(1) + T^* \lambda_k w_k(1)  = \lambda_k  h^0_k(1).
\end{cases}
\end{equation}
that is, via the linear transformation $\Xi$,
\begin{equation}\label{sys115}
\begin{cases}
-\bar w_{i,k}''(t) = 2T^*(\lambda_k + a_i)\lambda_k \bar h^0_{i,k} - (T^*)^2 \lambda_k^2 \bar w_{i,k} & \text{in $(0,1)$,} \\
\bar w_{i,k}(0)=0, \quad \bar w'_{i,k}(1) + T^* \lambda_k \bar w_{i,k}(1)  = \lambda_k \bar h^0_{i,k}(1).
\end{cases}
\end{equation}
We focus on the set of equations with $i=1$. Note that $\bar h^0_{1,k} \equiv 0$ for all $k \ge 2$, so $ \bar w_{1,k} \equiv 0$ for all $k \ge 2$. On the other hand, recall that $\eta(t) := \bar h^0_{1,1}(t) = \sin (T^*\sqrt{\lambda_1(-a_1 - \lambda_1)}\, t)$ satisfies
\begin{equation}\label{sys12}
\begin{cases}
\eta''(t) = (T^*)^2 (\lambda_1 + a_1)  \lambda_1 \eta(t) & \text{in $(0,1)$,} \\
\eta(0)=0, \quad \eta'(1) + T^* \lambda_1 \eta(1)  = 0,
\end{cases}
\end{equation}
We now multiply the equation in \eqref{sys115} (with $i, k=1$) by $\eta$, the equation in \eqref{sys12} by $\bar w_{1,1}$, integrate by parts on $(0,1)$ and sum to get
\[
-\bar w_{1,1}' \eta(1) + \bar w_{1,1}' \eta(0) + \bar w_{1,1} \eta'(1) - \bar w_{1,1} \eta'(0) = 2T^*(\lambda_1 + a_1)\lambda_1 \int_0^1  \eta^2 dt + (T^*)^2\lambda_1 a_1 \int_0^1 \bar w_{1,1} \eta dt.
\]
Using then the boundary conditions,
\[
-(T^*)^2\lambda_1 a_1 \int_0^1 \bar w_{1,1} \eta dt = 2T^*(\lambda_1 + a_1)\lambda_1 \int_0^1  \eta^2 dt +  \lambda_1 \eta^2(1).
\]
Evaluating the last expression, by \eqref{Tcond} one has
\[
-(T^*)^2 a_1 \int_0^1 \bar w_{1,1}(t) \sin (T^*\sqrt{\lambda_1(-a_1 - \lambda_1)}\, t) dt = T^*(\lambda_1 + a_1 ) < 0,
\]
so by Lemma \ref{rankL}
\[
\Lc'(T^*)[h^0] = w \notin R(\Lc(T^*)),
\]
that is the sufficient condition for $T^*$ to be a bifurcation point.

Finally, we observe that $\Gc$ is a compact operator on $X \times (0,\infty)$, while Theorem \ref{global_bif} requires it to be compact on $X \times \R$. We may circumvent this as follows. By Lemma \ref{unique}, \eqref{Gdef} has only the trivial solution for all $T \le \overline T$. In \eqref{Gdef}, we may then replace $T$ by $\varphi(T)$, where $\varphi : \R \to (0,\infty)$ is a positive  and increasing smooth function such that $\varphi(T) = T$ for all $T > \overline T$, so $\Gc$ becomes compact on $X \times \R$ and coincides with its original definition for all $T > \overline T$. Since there are no non-trivial solutions whenever $\varphi(T) \le \overline T$ (and no bifurcation points, i.e. $T^* > \overline T$), the previous analysis holds, and Theorem \ref{global_bif} applies. Moreover, if $(\sm, \tau)$ is a non-trivial solution of $\sm = \Gc(\sm, \tau)$, then $\varphi(\tau) > \overline T$, so $\varphi(\tau) = \tau$ and $(\sm, \tau)$ is really a solution to \eqref{Gdef}.

\end{proof}

\begin{figure}
\centering
\begin{tikzpicture}
      \draw[->] (-.2,0) -- (6.2,0) node[right] {$T$};
      \draw[->] (0,-.2) -- (0,2) node[left] {$\|(u,m)\|$};
       \draw[-, dashed] (0,.5) -- (5.7,.5) node[right] {$(\bar u, \bar m)$};
       \fill[red] (1.5,.5) circle[radius=1.5pt] node[below] {$T^*_1$};
      \draw[domain=.5:2,smooth,variable=\y,blue] plot ( {1.5+1.3*(\y-.5)*(\y)*(sin(660*(\y-.5)))^2 }, {\y} );
      \node at (5.2,1.4) {$\mathcal C_2$};
      \fill[red] (2.8,.5) circle[radius=1.5pt] node[below] {$T^*_2$};
     \fill[pattern=north west lines, pattern color=brown] (0,0) rectangle (.7,2);
    \draw	(.7,0) node[anchor=north] {$\overline T$};
     \draw[domain=0:1,smooth,variable=\t,blue] plot ( { 4+1.2*cos(180*\t*\t) }, { .5+.7*sin(180*\t) } );
     \node at (5.7,2) {$\mathcal C_1$};
     \fill[red] (5.2,.5) circle[radius=1.5pt] node[below] {$T^*_3$};
\end{tikzpicture}
\caption{\footnotesize A sample bifurcation diagram: the horizontal dashed line represents trivial solutions $(\bar u, \bar m)$, that exist for every $T > 0$. From the line of trivial solutions, branches (continua) $\mathcal C_n$ of non-trivial solutions emanate at certain values $T^*_n$ (in red). Theorem \ref{glob_bifo} states that a dichotomy occurs: {\it i)} branches may connect two different bifurcation points (as for $\mathcal C_2$), or {\it ii)} branches are unbounded (as for $\mathcal C_1$). Note that unbounded branches must contain sequences with $T \to \infty$ (see Proposition \ref{unb_bra}). Moreover, no non-trivial branch can enter the filled rectangle $X \times (0,\overline T)$. Indeed, in view of Lemma \ref{unique}, any non-trivial solution exists for $T > \overline T$ only. The reader may compare this sample diagram with diagrams in Figures \ref{fig_2branches}, \ref{fig_5branches}, \ref{fig_sch_br}, that have been obtained by numerical explorations of \eqref{MFG0}.  }\label{bif_ex}
\end{figure}
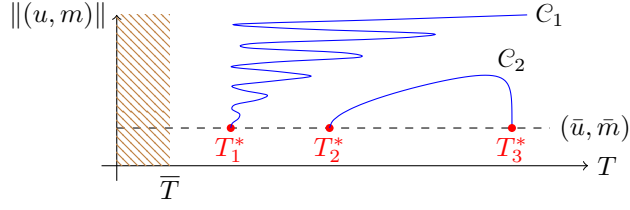

\begin{rem}\label{mfg_loc} It is straightforward to check that $\Gc$ is $C^2(X \times \R, X)$, so by Remark \ref{param_loc}, $\mathcal C$ is a continuously differentiable curve in a neighbourhood of $(0,T^*)$, parametrised by
\[
T = T^* + O(\eps), \quad (\sm_1, \sm_2) = (0,0) + \eps (\sm^*_1, \sm^*_2) + o(\eps),
\]
where $(\sm^*_1(x,t), \sm^*_2(x,t)) = \psi_{1}(x) \Xi \cdot \left[\sin \left(T\sqrt{\lambda_{1}(-a_1 - \lambda_{1})}\, t\right) , 0\right]^T$.

Going back to the unknowns $(m_1, m_2)$, this means that \eqref{MFG0} has a branch of nontrivial solutions
\[
T = T^* + O(\eps), \quad (m_1(x,s), m_2(x,s)) = (1,1) + \eps \psi_{1}(x) \Xi \cdot \left[\sin \left(\sqrt{\lambda_{1}(-a_1 - \lambda_{1})}\, s\right) , 0\right]^T + o(\eps).
\]

Note that, if $\sigma \neq 1$, the same arguments lead to bifurcation times $T^*$ that satisfy
\[
\tan \left(T^* \sqrt{\lambda_{1}(-a_1 - \sigma^2  \lambda_{1})}\right) = -\frac{1}{\sigma}\sqrt{\frac{-a_1- \sigma^2 \lambda_{1}}{\lambda_{1}}},
\]
with branches having local parametrisation
\[
(m_1(x,s), m_2(x,s)) = (1,1) + \eps \psi_{1}(x) \Xi^{-1} \left[\sin \left({\sqrt{\lambda_{1}(-a_1 - \sigma^2 \lambda_{1})}}\, s \right) , 0\right]^T + o(\eps).
\]

\end{rem}

We are now in the position to prove the main global bifurcation result of this paper.

\begin{proof}[Proof of Theorem \ref{glob_bifo_mfg}] The statement directly follows by Theorem \ref{glob_bifo} and Remark \ref{mfg_loc}, noting that $(u, m) \mapsto (\su, \sm)$ is a continuous bijection on $X \times (0, \infty)$. The bifurcation condition \eqref{TC} is simply derived by taking the inverse of $z \mapsto \tan z$ in \eqref{Tcond}. \end{proof}

Finally, by parabolic regularity we can show that if $\mathcal C$ is an unbounded branch, that is if $\mathcal C$ does not connect two different bifurcation points, then it must be unbounded in the direction $T \to +\infty$, in the following sense (see also Figure \ref{bif_ex}).

\begin{prop}\label{unb_bra} Let $\mathcal C$ be as in Theorem \ref{glob_bifo}. Then,
\begin{itemize}
\item[{\it i)}] $\mathcal C$ contains some $(0,T_1)$, where $T^* \neq T_1$, or
\item[{\it ii)}] $\mathcal C$ contains sequences with $T \to +\infty$.
\end{itemize}
\end{prop}

\begin{proof} By contradiction, suppose that both {\it i)} and {\it ii)} are false, in particular that there exists $\widehat T > 0$ such that for all $(\sm, T) \in \mathcal C$, $T \le \widehat T$. Since $ \mathcal C$ consists of non-trivial solutions, it also holds true that $T \in (\overline T, \widehat T)$ by Lemma \ref{unique}. Since $\sm = \Gc(\sm, T)$, arguing as in the first part of the proof of Lemma \ref{compcont} we conclude by parabolic regularity that $\sm$ belongs to a bounded set of $X$ (recall that \eqref{Vass} is in force, in particulart $V_i$ are bounded on $\R^2$), so $ \mathcal C$ is bounded in $X \times \R$. This contradicts Theorem \ref{glob_bifo}, since if {\it i)} is false, then $ \mathcal C$ must be unbounded. \end{proof}

\subsection{Additional remarks}

\begin{rem}\label{addrem1} \textit{On the assumption $\sigma^2 \lambda_1 < -a_1 < \sigma^2 \lambda_2$ of Theorem \ref{glob_bifo_mfg}}.

As we explained earlier, bifurcation times $T^*$ are so that the linearized system \eqref{diff_zero} has a one dimensional vector space of non-trivial solutions, that is, $N(\Lc(T))$ is spanned by a single element. In Lemma \ref{whoisker} it is proved that this happens whenever $T$ satisfies \eqref{kernel_eq}; if $\sigma^2 \lambda_1 < -a_1 < \sigma^2 \lambda_2$, this equation has solutions if only if $k_j = 1$. Suppose now that
\[
-a_1 \neq  \sigma^2 \lambda_k \quad \forall k \in \mathbb{N},
\]
and, if $K$ is such that $\sigma^2 \lambda_K < -a_1 < \sigma^2 \lambda_{K+1}$, then
\[
\lambda_1, \ldots, \lambda_K \text{ are simple.}
\]

Let $\sigma = 1$ for simplicity. Then, for any $k = 1, \ldots, K$ fixed, \eqref{kernel_eq} has a family of solutions of the form
\[
T^*_{n,k} = \frac{1}{\sqrt{\lambda_{k}(-a_1 -\lambda_{k})}}\left[n\pi - \arctan\left(\sqrt{\frac{-a_1- \lambda_{k}}{\lambda_{k}}}\right)\right].
\]
Of course, $T^*_{n,k} \neq T^*_{m,k}$ for all $n \neq m$. If it is also true that for some $(n, k)$, $T^*_{n,k} \neq T^*_{m,j}$ for all $j = 1, \ldots, K$, $j \neq k$, $m \ge 1$, then $T^*_{n,k}$ becomes a bifurcation point. Indeed, in this case $N(\Lc(T))$ is still one dimensional, and the arguments and conclusions of Theorem \ref{glob_bifo} hold. In particular, there exists a continuum $C_{n, k}$ branching off from the trivial solution which is either unbounded or meets again the trivial solution in another bifurcation point. An example of this scenario will be discussed later (see Section \ref{exp2}).

\end{rem}

\begin{rem}\label{addrem2} \textit{The non-quadratic case.}

The same bifurcation methods apply for more general MFG systems with non-quadratic Hamiltonians $H_i (x, p)$ of the form
\begin{equation}\label{MFGnq}
\begin{cases}
-\partial_t u_i - \sigma \Delta u_i + H_i (x, \nabla u_i) = V_i(m_1, m_2),& \text{in $Q_T, \, i=1,2$,} \\
\partial_t m_i - \sigma \Delta m_i - \diverg(\nabla_p H_i (x, \nabla u_i) \, m_i)  = 0,
\end{cases}
\end{equation}
under the assumption that for some $c_H, \kappa_{1,2} > 0$ and for all $x \in \Omega$,
\[
\begin{split}
& \textit{i) } \nabla_p H_i(x, 0) = 0, \\
&  \textit{ii) } D^2_p H_i(x, 0) = \kappa_i I, \\
&  \textit{iii) }  |H_i(x, p)| \le c_H (1 + |p|^2), \, |p|^2 |\nabla_p H_i(x, p)| +  |\nabla_x H_i(x, p)|\le c_H (1 + |p|^3) \quad \forall p.
\end{split} 
\]
While  \textit{iii)} is related to a priori estimates needed to have $\Gc$ well-defined and regular, \textit{i)} - \textit{ii)} imply that the linearized system is identical to the linearized system of a quadratic problem of the form \eqref{MFG0}. Indeed, the linearized operator $z = D_\sm \Gc(0, T)[h]$ associated to the linearization of \eqref{MFGnq} reads
\begin{equation}\label{diff_zero_nq}
\begin{cases}
- \partial_t v_i - \sigma T \Delta v_i = T[\partial_{m_1} V_i(1, 1) h_1 + \partial_{m_2} V_i(1, 1)h_2],& \text{in $Q$,} \\
\partial_t z_i -  \sigma T \Delta z_i - \kappa_i T  \Delta v_i  = 0, \\
\end{cases}
\end{equation}
and an identical analysis can be carried over with the additional parameters $k_i$. Let us discuss briefly the effect of these parameters in a simple model, i.e. a single population setting where $\kappa_1 = \kappa_2 = \kappa > 0$,
\[
V_1(m_1, m_2) = -a m_1, \qquad V_2(m_1, m_2) = 0, \qquad a > 0.
\]
Arguing as in the previous section, a necessary condition for $T_n$ to be a bifurcation point is that
\[
T_n(\kappa) = \frac{1}{\sqrt{\lambda_{1}(\kappa a - \sigma^2 \lambda_{1})}}\left[n\pi - \arctan\left(\frac{1}{\sigma}\sqrt{\frac{\kappa a- \sigma^2 \lambda_{1}}{\lambda_{1}}}\right)\right]
\]
If $\kappa$ is such that $\sigma^2 \lambda_1 < \kappa a < \sigma^2 \lambda_2$, we indeed obtain branches of non-trivial solutions. For this to happen, $k$ has to be bounded away from zero. If $k$ is large, one may argue as in Remark \ref{addrem1}, namely bifurcation of non-trivial solutions from $T_n(\kappa)$ shows up if some kind of non-resonance is guaranteed; note that, for any $n$ fixed, $T_n(\kappa) \to 0$ as $\kappa \to \infty$. In other words, one may find branches of non-trivial solutions for small time horizons as soon as $D^2 H_i(0)$ is large.

\end{rem}

\begin{rem}\label{addrem3} \textit{On the assumption \eqref{Vass}}.

In the same spirit of Remark \ref{addrem2}, what really matters for our bifurcation analysis is the linearization of \eqref{MFG0} around the trivial solution $(\bar u, \bar m)$. The analysis carried out in Section \ref{s_bif} is identical if one replaces \eqref{Vass} by 
\begin{equation}\label{Vass2}
\begin{split}
\bullet &\, \text{$V_i = V_i(x,m) \in C^\infty(\Omega \times (0,+\infty) \times (0,+\infty))$ are bounded}, \\
\bullet &\, \text{there exist $a_1 \le a_2$ with $a_1 < 0$, and an invertible $2 \times 2$ matrix $\Xi$ such that for all $x \in \Omega$}, \\
&
\begin{bmatrix}
a_1 &  0 \\
  0 &  a_2
\end{bmatrix}
= \Xi \cdot J_m V(x, 1, 1) \cdot \Xi^{-1}.
\end{split}
\end{equation}
Indeed, the linearized system does not change at all under the additional $x$-dependance of $V_i$ if the eigenvalues and eigenvectors of $J_m V(x, 1, 1)$ do not vary on $\Omega$, and $\Xi$ decouples the linearized system into two systems that can be treated as \eqref{sys95}. The cases of $J V(1,1)$ having a single eigenvalue with geometric multiplicity equal to one or complex eigenvalues are more delicate and require different treatments.

We note that adding an $x$-dependance does not change the qualitative behaviour of branches close to bifurcation points, but may alterate significantly the system as soon as $(u,m)$ differs from the trivial solution $(\bar u, \bar m)$.
\end{rem}

\begin{rem}\label{addrem4} \textit{More than two populations}.

The present work deals with two-population models, but the arguments presented could be adapted to more general systems of the form
\[
\begin{cases}
-\partial_t u_i - \sigma \Delta u_i + \frac{1}{2}|\nabla u_i|^2 = V_i(m_1, m_2, \ldots, m_M), \\
\partial_t m_i - \sigma \Delta m_i - \diverg(\nabla u_i\, m_i)  = 0, & \text{in $Q_T, i =1, \ldots, M$,}
\end{cases}
\]
as soon as $J V(1,\ldots,1)$ has at least one negative eigenvalue. This framework should be even more rich, and will be matter of future work.
\end{rem}

\section{Two applications}\label{s_appl}

In this section we consider two different models. For both models, we apply Theorem \ref{glob_bifo} to obtain the existence of branches of non-trivial solutions, with their explicit parametrization close to bifurcation times. We then present some numerical experiments to show how solutions behave along selected branches as $T$ varies. Such analysis has just the aim of observing the behaviour of solutions for particular choices of the parameters, and by no means is meant to give a general description of the qualitative properties of solutions to multi-population MFG systems. Still, the following results highlight different scenarios.

To simplify the numerical analysis, and better compare the theoretical and experimental sides, we restrict to space dimension $N = 1$, i.e.
\[
\Omega = (0,1).
\]
In this setting, Neumann eigenvalues and eigenvectors of $-\Delta$ are explicit, namely we have
\[
\lambda_k = (k\pi)^2 \qquad \psi_k(x) = \cos(k \pi x),
\]
and all the eigenvalues are simple.

Numerical solutions will be obtained by finite difference methods. We use the techniques described in \cite{ABC}, that rely on monotone approximations of the Hamiltonian and on a suitable weak formulation of the Fokker-Planck equation. Since \eqref{MFG0} couples forward and backward in time equations, it cannot be solved by merely marching in time, so we implement a Newton-Raphson method for the whole system on $Q_T$, mimicking the infinite dimensional equation $\sm = \Gc(\sm, T)$. For additional details we refer to \cite[Section 5] {ABC} and references therein. We mention that to get a precise approximation of \eqref{MFG0} close to bifurcation points, we need a space-time grid with rather small steps (see the discussion in Experiment 1); here, we will use a uniform $400 \times 400$ grid. Theorem \ref{glob_bifo_mfg} suggests that \eqref{MFG0} has a rich structure of solutions; the numerical method is indeed very sensitive to the initial guess $m^0$, that has to be properly chosen to select desired branches.

We finally mention that the radial case $\Omega = B_R$ could be treated in the same ways, both from the theoretical and numerical sides.

\subsection{A single population model with aggregation}\label{spa}

We consider a single population model where
\[
V_1(m_1, m_2) = -a \, g(m_1), \qquad V_2(m_1, m_2) = 0, \qquad a > 0,
\]
where $g(m)$ is a smooth increasing function such that $g(m) = m$ for all $m \in [0,M]$ (here $M > 1$ is fixed) and $g$ is bounded on $\R$ (we will see that for suitable $M >> 1$,  solutions have $m$ components such that $\|m\|_\infty \le M$ independently on $T$, so $u,m$ will really be solutions of the model problem with linear cost $V_1(m_1, m_2) = -a m_1$). The cost $V_2$ is trivial, hence we restrict our attention to the components $(u_1, m_1)$; here, the system is decoupled and the second population is not playing any significant role. The coupling $V_1(\cdot, m_2)$ is decreasing, namely players of the first population are attracted toward congested areas. Since
\[
JV := JV(1,1) = 
\begin{bmatrix}
 -a &  0 \\
  0 &  0
\end{bmatrix},
\]
we can rephrase Theorem \ref{glob_bifo_mfg} as follows:
\begin{thm}\label{glob_bifo_aggreg} Suppose that $(\sigma \pi)^2 < a < 4( \sigma \pi)^2$, and that $T^*_n$ satisfies
\[
T^*_n = -\frac{1}{\pi \sqrt{a- (\sigma \pi)^2} }\arctan\left( \frac{\sqrt{a- (\sigma \pi)^2}}{\sigma \pi} \right) + \frac{n}{\sqrt{a- (\sigma \pi)^2}}
\]
for some $n \in \mathbb{N}$. Then, $(\bar u, \bar m, T^*_n) \in S$. Let $C_n$ be the connected component of $S$ to which $(\bar u, \bar m,T^*_n)$ belongs. Then $C_n$ contains some $(\bar u, \bar m, \widetilde T)$, where $T^* \neq \widetilde T $, or $C_n$ is unbounded.

Finally, $C_n$ is a continuously differentiable curve in a neighbourhood of $T^*_n$, parametrized by
\begin{equation}\label{para}
T = T^*_n + O(\eps), \quad m_1(x,s) = 1 + \eps \cos(\pi x) \sin \left(\pi \sqrt{a- (\sigma \pi)^2} \, s\right) + o(\eps).
\end{equation}
\end{thm}

Let us comment on the previous result, in particular on the expansion \eqref{para}. The theorem provides the existence, for $T$ close to $T^*_n$ of a solution to \eqref{MFG0} such that $m$ can be written as in \eqref{para}: qualitatively, $m_1$ is a perturbation of the constant distribution $\bar m \equiv 1$ that is approximately of the form $ \cos(\pi x) \sin \left(2\pi s/\tau\right) $, where the oscillation period (in time) is $\tau = 2 (a- (\sigma \pi)^2)^{-1/2}$. Let us focus on the time dependance: the final time $T = T^*_n + o(\eps)$ has the form
\[
T = \frac{\tau}{2}(n- \delta) + o(\eps),
\]
where $\delta < 1/2$ is some constant depending on the data. Therefore, if $n=1$, $s \mapsto  \sin \left(2\pi s/\tau\right)$ reaches a maximum and decreases until $s = T^*_1$.  If $n=2$, $s \mapsto  \sin \left(2\pi s/\tau\right)$ reaches a maximum, decreases to a minimum and increases again up to $s = T^*_2$. In general, $s \mapsto  \sin \left(2\pi s/\tau\right)$ switches $n-1$ times between maxima and minima during the whole time interval $[0, T^*_n]$, see Figure \ref{sins}. This means that $m_1$ switches $n-1$ times between the profiles $1 + \eps \cos(\pi x)$ and $1 - \eps \cos(\pi x)$ as $s$ increases, or in other words, $m_1$ performs (a bit less than) $n/2$ full oscillations in time around the constant state.

\begin{figure}
\centering
\includegraphics[width=5cm]{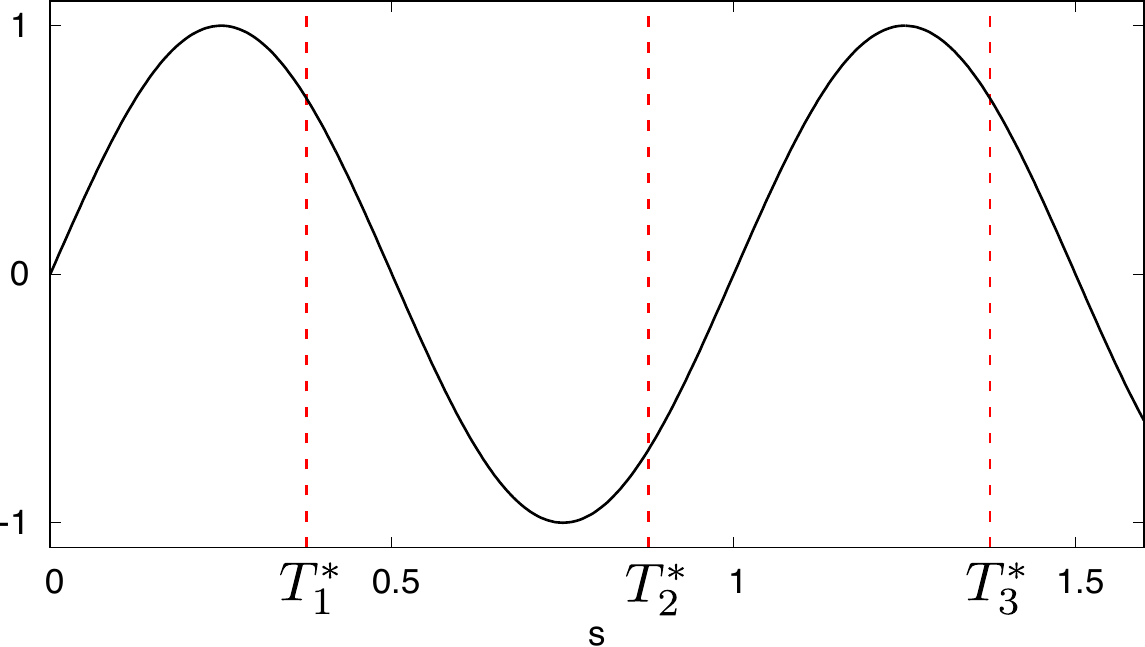}
\caption{\footnotesize Plot of the function $s \mapsto  \sin \left(2\pi s/\tau\right)$, $\tau = 1$, and bifurcation times $T^*_1, T^*_2, T^*_3$.}\label{sins}
\end{figure}

\smallskip

This can be clearly observed in the following numerical experiments. 

\noindent {\bf Experiment 1.} We choose the parameters
\[
\sigma = \frac{1}{\pi}, \qquad a = 2,
\]
so that the assumptions of Theorem \ref{glob_bifo_aggreg} are satisfied. We obtain several families of numerical solutions in the following way: for any fixed $n \in \mathbb{N}$, we expect the existence of a solution of the form $m(x, s) = \tilde m(x,s) + o(\eps)$, with $\tilde m(x,s) = 1 + \eps \cos(\pi x) \sin \left(\pi \sqrt{a- (\sigma \pi)^2} \, s\right)$, for $T$ close to $T^*_n$. We choose $\tilde m$ as the initial guess of the Newton method, with $\eps$ of order $10^{-1}$ (or $10^{-2}$), and increase $T$ starting from $T^*_n$ until the numerical method converges to a non-trivial solution. We then ``follow'' the branch by continuation with respect to the parameter $T$, as in Figure \ref{fig_2branches}.

\begin{figure}
\centering
\includegraphics[width=6cm]{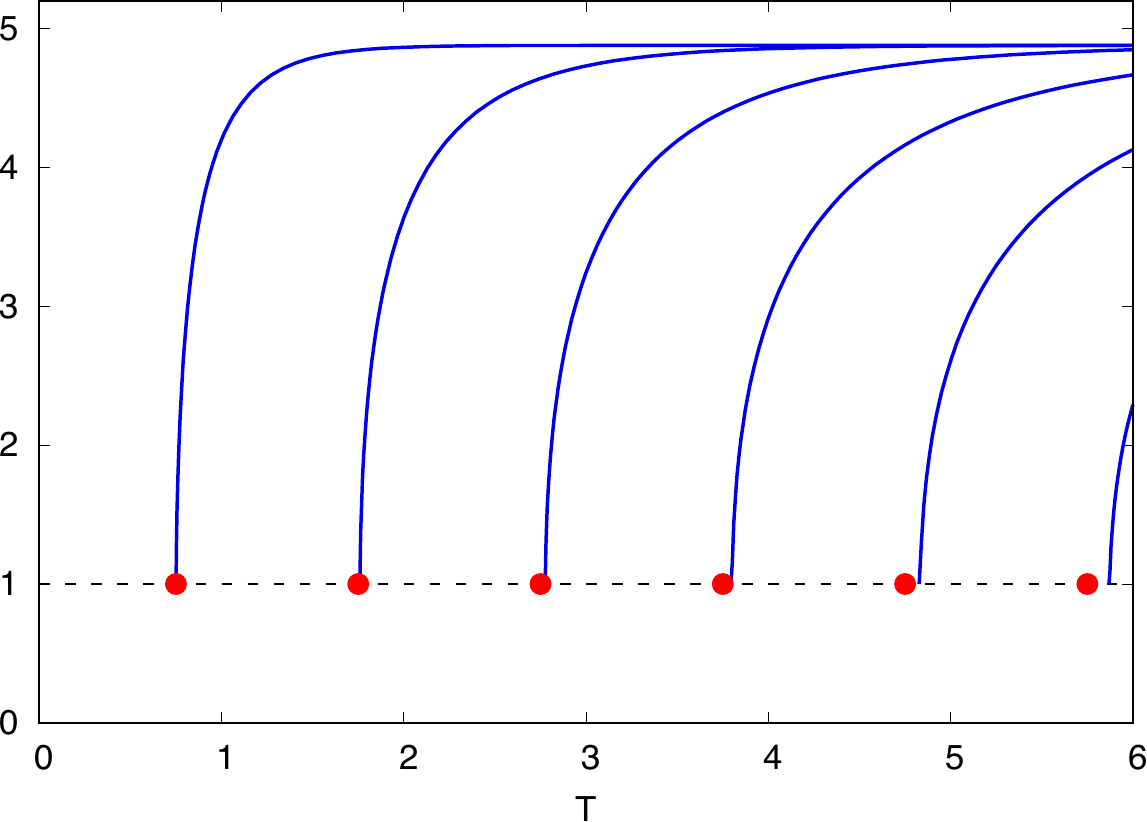}
\caption{\footnotesize Branches of non-trivial solutions, corresponding to $C_1 \ldots, C_6$. On the $y$-axis the value of $\|m\|_{\infty}$ is plotted. Red circles are the expected bifurcation points from the trivial solution $\bar m \equiv 1$.}\label{fig_2branches}
\end{figure}

Typically, the $\sup$-norm of $m_1$ is an increasing function of $T$. On one hand, the parameter $T$ can be decreased up to some $\overline T_n$ such that the corresponding solution $m_1$ converges uniformly to the trivial state; as $T$ approaches $\overline T_n$, one observes the expected qualitative behaviour in terms of space-time oscillations, in accordance with \eqref{para}. See Figure \ref{fig_2b}. In principle, $\overline T_n$ should almost coincide with the theoretical bifurcation times $T^*_n$, but those values are slightly different in general (see again Figure \ref{fig_2branches}). We believe that this can be explained as follows: $T^*_n$ are such that the linearized system \eqref{diff_zero} has a non-trivial solution, and this happens whenever the second order (in time) system \eqref{sys9} is satisfied. Note that \eqref{sys9} involves also the (space) eigenvalue $\lambda_k$. It is known that finite difference eigenvalues of the Laplacian do not coincide with eigenvalues of its continuous counterpart, though the former converge to the latter as the mesh step decreases. Therefore, bifurcation times of the discretised problem can differ from the continuous one, because of the gap between discrete and continuous eigenvalues, that is what we observe in our numerical results. However, as space and time steps decrease, we note that $\overline T_n$ indeed converge to the theoretical values  $T^*_n$.

\begin{figure}
\centering
\includegraphics[width=4.5cm]{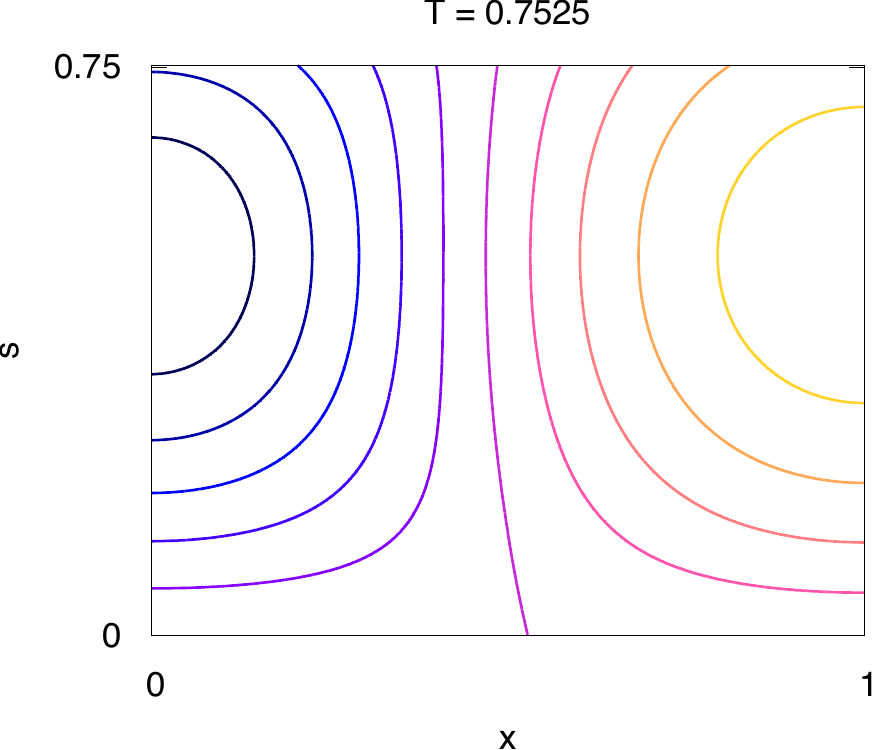}
\includegraphics[width=4.5cm]{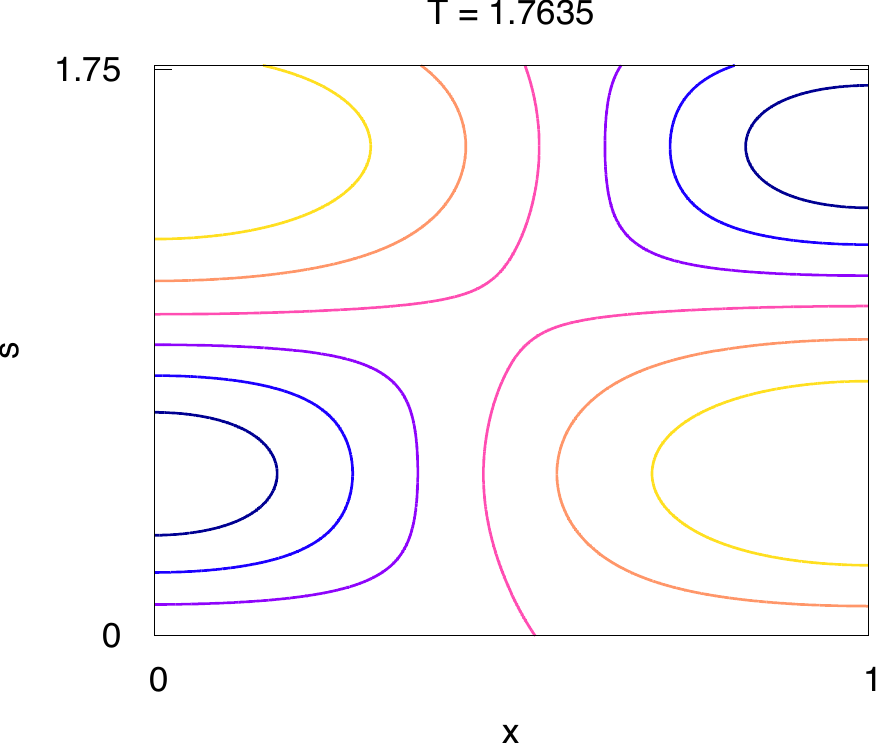}
\includegraphics[width=4.5cm]{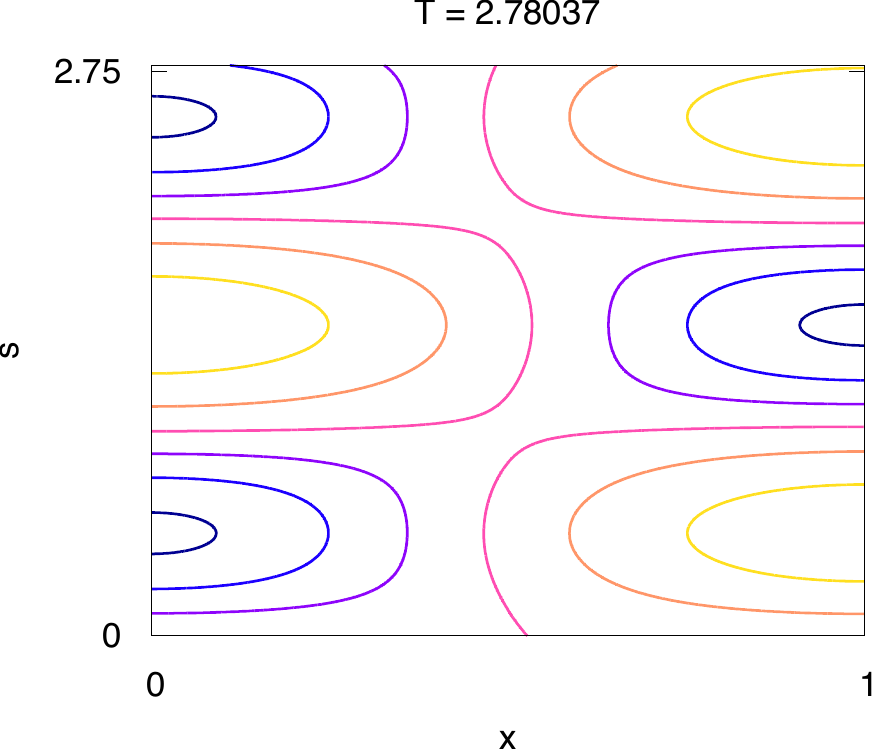}
\includegraphics[width=4.5cm]{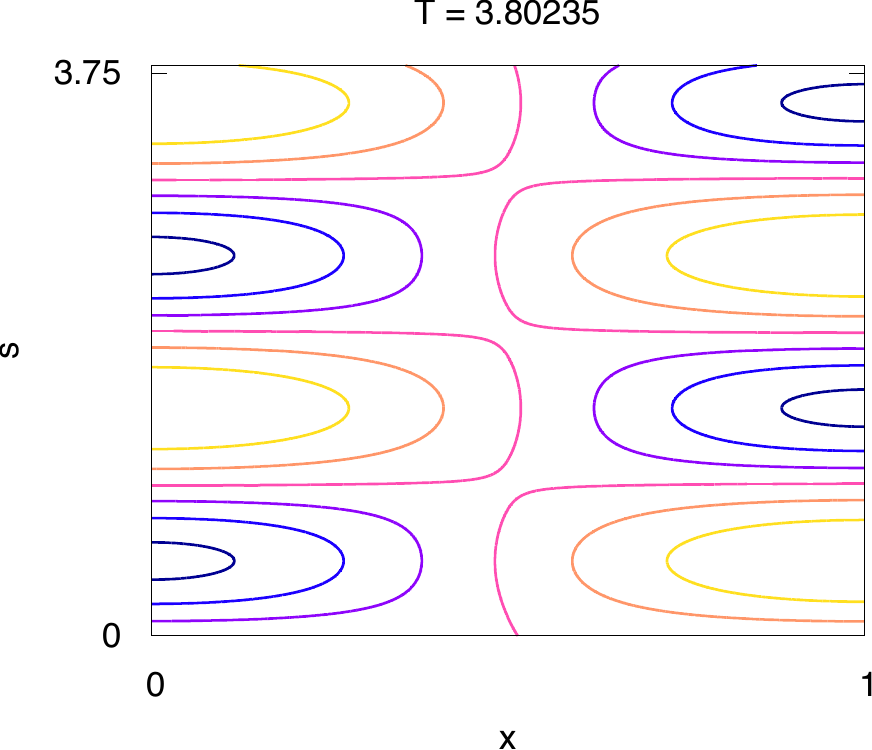}
\includegraphics[width=4.5cm]{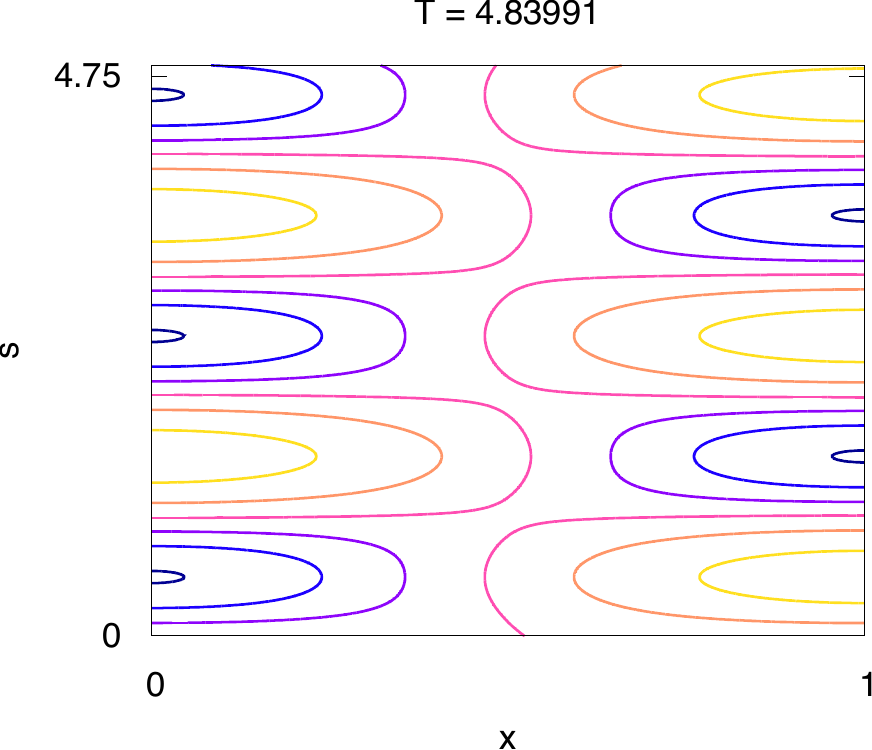}
\includegraphics[width=4.5cm]{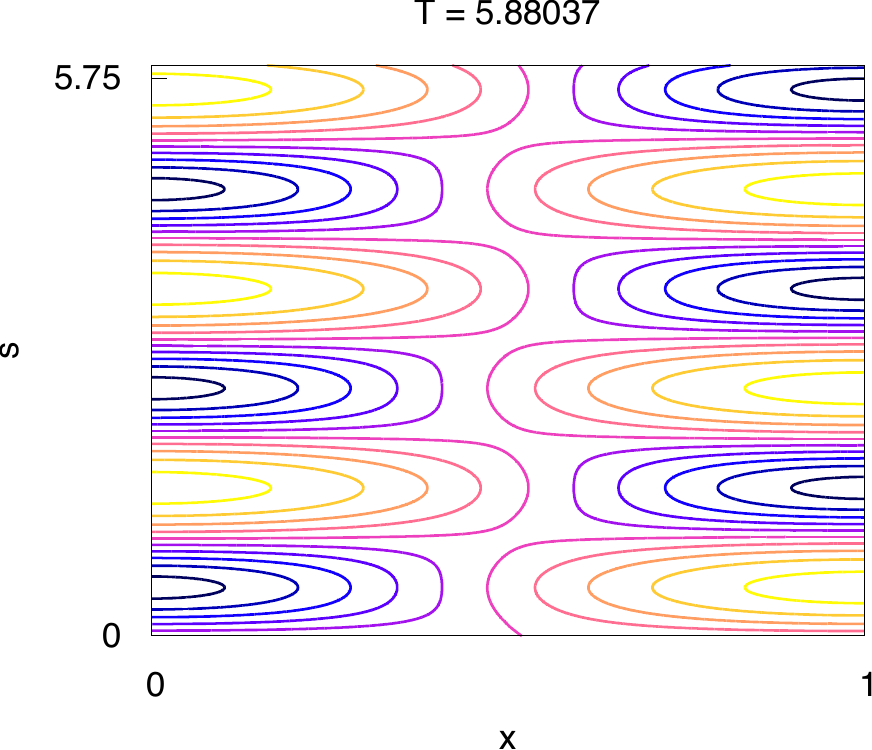}

\caption{\footnotesize From top-left to bottom-right, (space-time) contours of $m_1$ belonging to $C_1 \ldots, C_6$ respectively, for $T$ close to the bifurcation times $T_n^*$. Colours black-blue-violet-yellow vary from $\max_{Q_T} m_1$ to $\min_{Q_T} m_1$.}\label{fig_2b}
\end{figure}

\begin{figure}
\centering
\includegraphics[width=4.5cm]{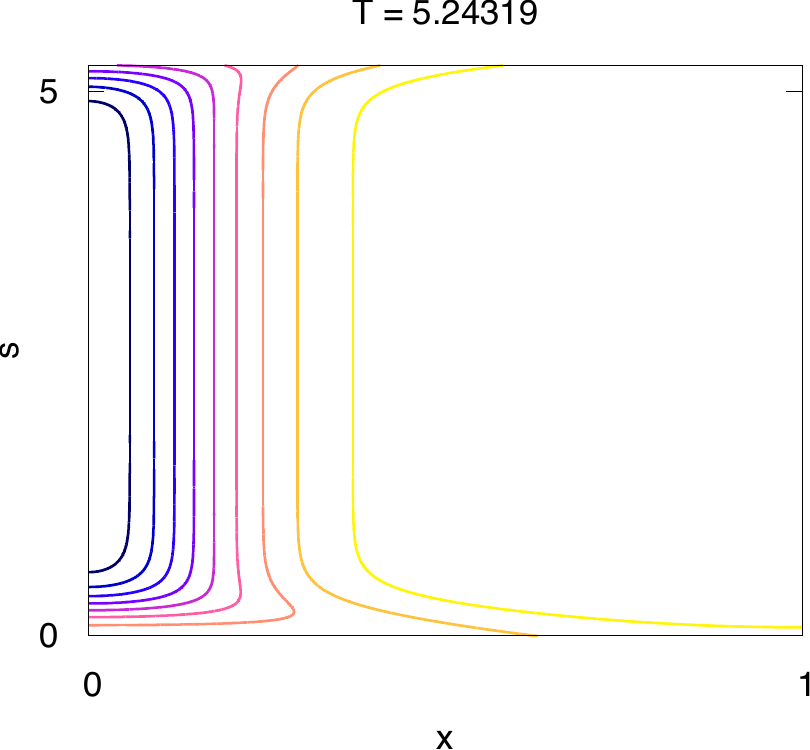}
\includegraphics[width=4.5cm]{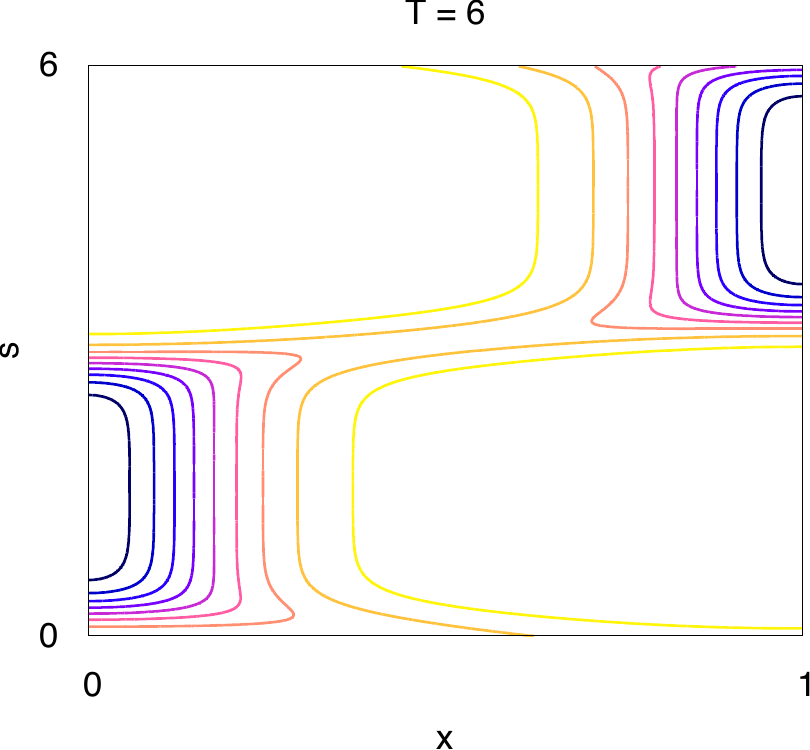}
\includegraphics[width=4.5cm]{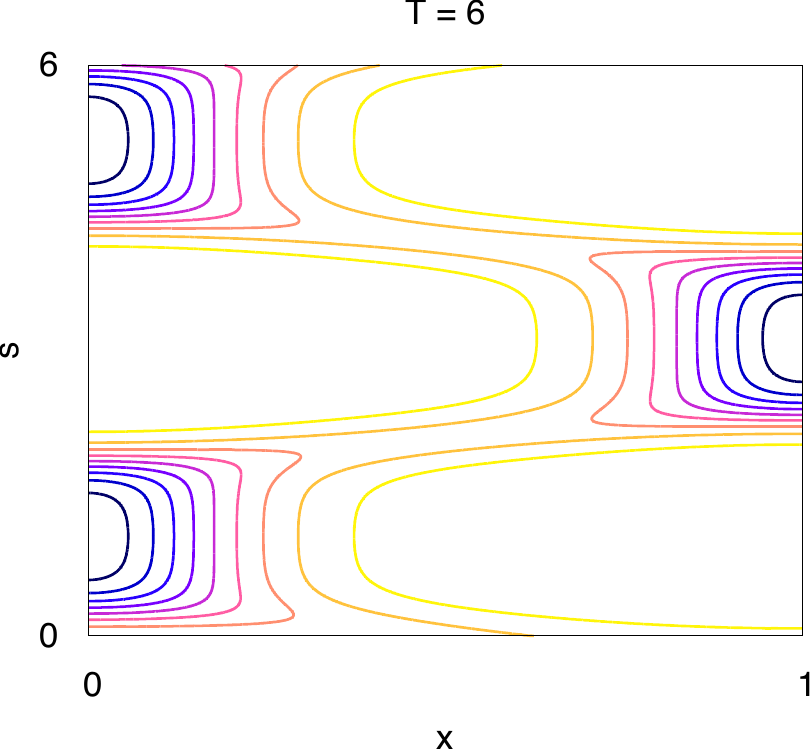}

\caption{\footnotesize From left to right, (space-time) contours of $m_1$ belonging to $C_1 \ldots, C_3$ respectively, for large values of $T$. Colours black-blue-violet-yellow vary from $\max_{Q_T} m_1$ to $\min_{Q_T} m_1$.}\label{fig_2f}
\end{figure}

On the other hand, $T$ can be increased arbitrarily. We are not able to carry out a qualitative theoretical analysis on the behaviour of solutions far from the bifurcation points, but we observe the following phenomenon in our numerical results. Along the first branch $C_1$, $m_1$ evolves quickly into a stable state, until the final horizon $T$ (here $\max_{[0,1]} m_1(\cdot, s)$ slightly decreases as $s$ approaches $T$ because the value function $u$ has to achieve the identically zero final cost). This asymptotic behaviour can be also observed in Figure \ref{fig_2branches}, as $\|m_1\|_{L^\infty(Q_T)}$ reaches a constant in time value as $T$ increases. The second branch $C_2$, for $T$ large, consists of solutions that are stable up to time (approximately) $T/2$, and switch quickly to another stable state up to final time. Similarly, in $C_3$ solutions switch two times, the third stable state being identical to the first one. As $T$ increases, $m_1$ becomes substantially stable in time, but shows $n-1$ rapid changes of state; apparently, the number of oscillations in time does not vary on $C_n$. We believe that such stable states are solutions of the stationary MFG
\[
\begin{cases}
 - \sigma \Delta u_i + \frac{1}{2}|\nabla u_i|^2 + \overline H_i = V_i(m_1(x), m_2(x)),& \text{in $\Omega$, i =1,2,} \\
 - \sigma \Delta m_i - \diverg(\nabla u_i\, m_i)  = 0, \quad \int_{\Omega} m_i = 1.
\end{cases}
\]

\smallskip \noindent {\bf Experiment 2.}\label{exp2} We now choose the parameters
\[
\sigma = \frac{1}{\pi}, \qquad a = 5.
\]
Note that $\pi^2 = \lambda_1 < 4 \pi^2 = \lambda_2 < a < 9 \pi^2 = \lambda_3$; therefore Theorem \ref{glob_bifo_aggreg} does not apply directly, but we may argue as in Remark \ref{addrem1}. Indeed, $\lambda_1$ and $\lambda_2$ are clearly simple eigenvalues. Therefore, we expect the existence of two families of non-trivial solutions branching off from
\[
\begin{split}
&T^*_{n,1} = \frac{n}{2} - \frac{1}{2\pi} \arctan(2), \\
&T^*_{n, 2} = \frac{n}{2} - \frac{1}{2\pi} \arctan\left(\frac{1}{2}\right),
\end{split}
\]
since $T^*_{n,1} \neq T^*_{m,2}$ for all $n, m$. Close to $T^*_{n,k}$ we expect to find solutions of the form \[m(x,s) \approx 1 + \eps \cos(k \pi x) \sin \left(2 \pi s\right).\] We represent some of these families in Figures \ref{fig_5branches} and \ref{fig_5b}. Branches behave qualitatively as in Experiment 1: the space-time structure is preserved as $T$ varies and the long-time regime is analogous. Note that as $T$ approaches the bifurcation parameter, in order to continue the branch up the the bifurcation point one has to decrease $T$ very slowly; in other words, a small perturbation of $T$ induces on the corresponding solution a significant variation. This is particularly noticeable for large values of $a$ and for the first branches $T_{1,1}$, $T_{1,2}$, etc.

\begin{figure}
\centering
\includegraphics[width=7cm]{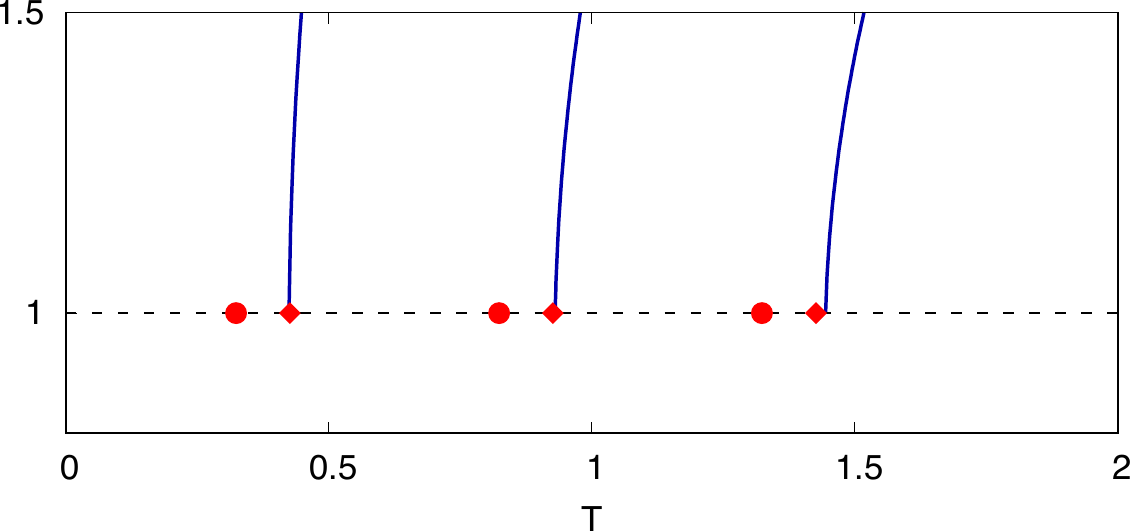}
\caption{\footnotesize Branches of non-trivial solutions, corresponding to $C_{1,2}, C_{2,2}, C_{3,2}$. On the $y$-axis the value of $\|m\|_{\infty}$ is plotted. Red circles and diamonds are the expected bifurcation points $T_{n,1}^*$ and $T_{n,2}^*$ respectively from the trivial solution $\bar m \equiv 1$.}\label{fig_5branches}
\end{figure}

\begin{figure}
\centering
\includegraphics[width=4.5cm]{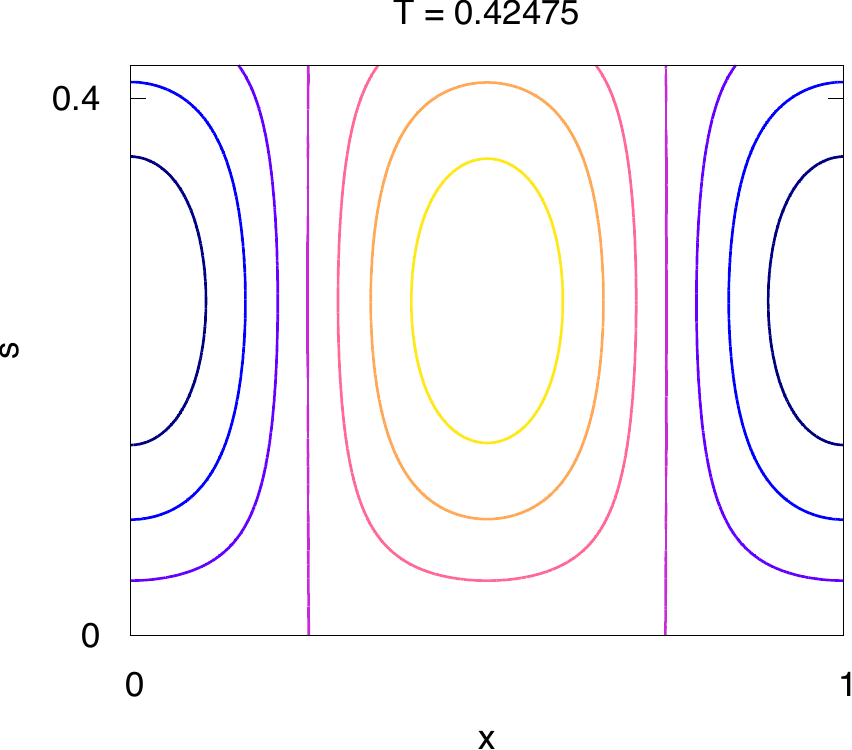}
\includegraphics[width=4.5cm]{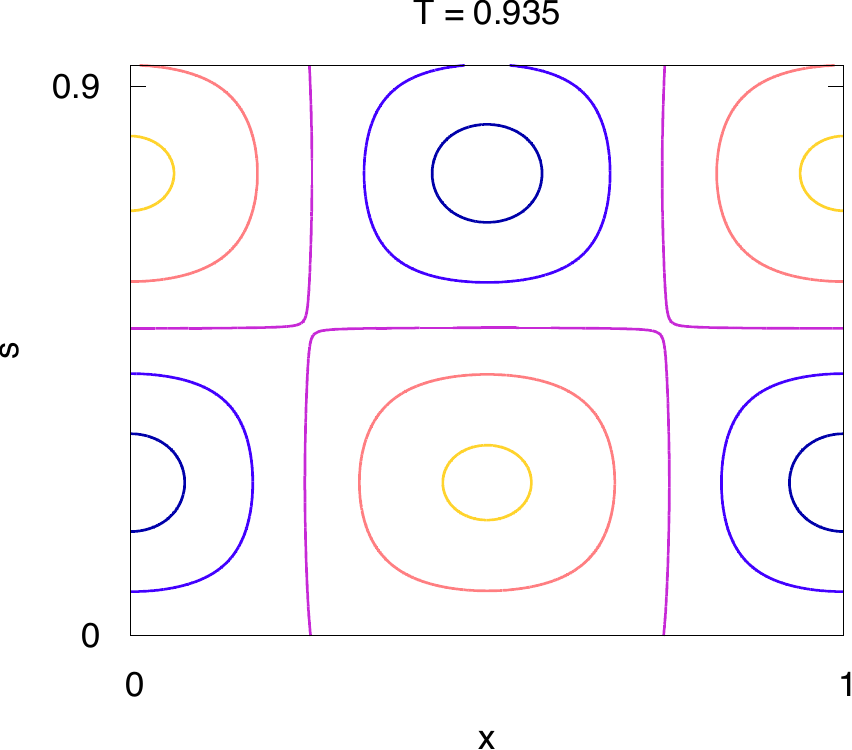}
\includegraphics[width=4.5cm]{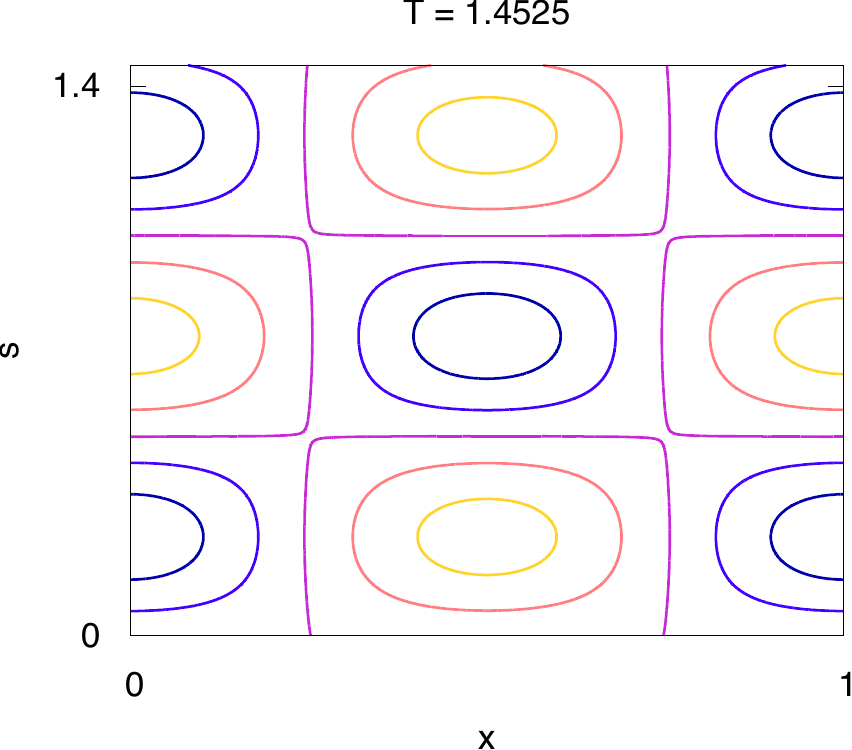}

\caption{\footnotesize From left to right, (space-time) contours of $m_1$ belonging to $C_{1,2}, C_{2,2}, C_{3,2}$ respectively, close to bifurcation times $T^*_{n, k}$. Colours black-blue-violet-yellow vary from $\max_{Q_T} m_1$ to $\min_{Q_T} m_1$.}\label{fig_5b}
\end{figure}

\smallskip \noindent {\bf Experiment 3.} We go back to the initial choice of the parameters
\[
\sigma = \frac{1}{\pi}, \qquad a = 2,
\]
but consider a MFG system with non-quadratic Hamiltonian, i.e. we discretize \eqref{MFGnq} with $H_i$ of the form
\[
H_i(p) = \frac{1}{2} |p|^\gamma, \qquad \gamma > 1.
\]
Note that $H_i$ is not even $C^2$ in a neighborhood of zero if $\gamma \in (0,2)$, so the first variation of  \eqref{MFGnq} contains a singular part. On the other hand, if $\gamma > 2$ the linearization of \eqref{MFGnq} is well-defined, but $D^2 H_i(0) =: \kappa = 0$. Therefore, arguing as in Remark \ref{addrem2}, the linearized system has no non-trivial solutions for all $T$, so we do not expect the existence of branches bifurcating from the trivial solution. Still, if $\gamma$ is close to two, \eqref{MFGnq} can be regarded as a small perturbation of the quadratic problem \eqref{MFG1}, and one might expect to find some analogies between the two cases.

The strategy to obtain non-trivial solutions is as follows: we start the Newton method using solutions of the quadratic problem as initial guesses. This method is particularly efficient if $T$ is not too close to the bifurcation values of the quadratic problem. Once a (discrete) solution is found, we continue along the parameter $T$. The sub-quadratic ($\gamma < 2$) and super-quadratic ($\gamma > 2$) regimes show dramatically different behaviours, that are represented in Figure \ref{nonqbranches}.

If $\gamma = 1.9$, for large values of $T$ solutions are somehow close to solutions of the quadratic problem. As $T$ decreases, branches do not collapse at some $T^*$ to the trivial state, but seem to continue up to $T = 0$. Actually, this continuation is hard to be carried over to arbitrarily small time horizons, as $m$ goes to the trivial state very quickly in the $\sup$-norm. Still, $m_1$ remains numerically bounded away from one even for $T$ very small. This property is shared between all branches, indicating that there might be a clustering of an infinite number of branches as $T \to 0$ to the trivial solution.

If $\gamma = 2.1$, a branch of stable (in time) solutions behaving as solutions of the quadratic problem is found for large values of $T$. As $T$ decreases, this branch reaches a turning point and continues back towards the direction $T \to \infty$. In this direction, $m_T$ converges to the trivial solution as $T$ goes to infinity, but remains non-trivial for all $T$. In this case, branches cluster around the trivial solution as $T \to \infty$.

Recall that if $H_i \in C^2(\RsetN)$ one has uniqueness of solutions for $T$ small, by reasoning as in Lemma \ref{unique}. This is consistent with numerical results for the case $\gamma > 2$. When $\gamma < 2$, so $H_i \notin C^2(\RsetN)$, numerical simulations suggest that uniqueness for small $T$ may fail.

\begin{figure}
\centering
\includegraphics[width=6cm]{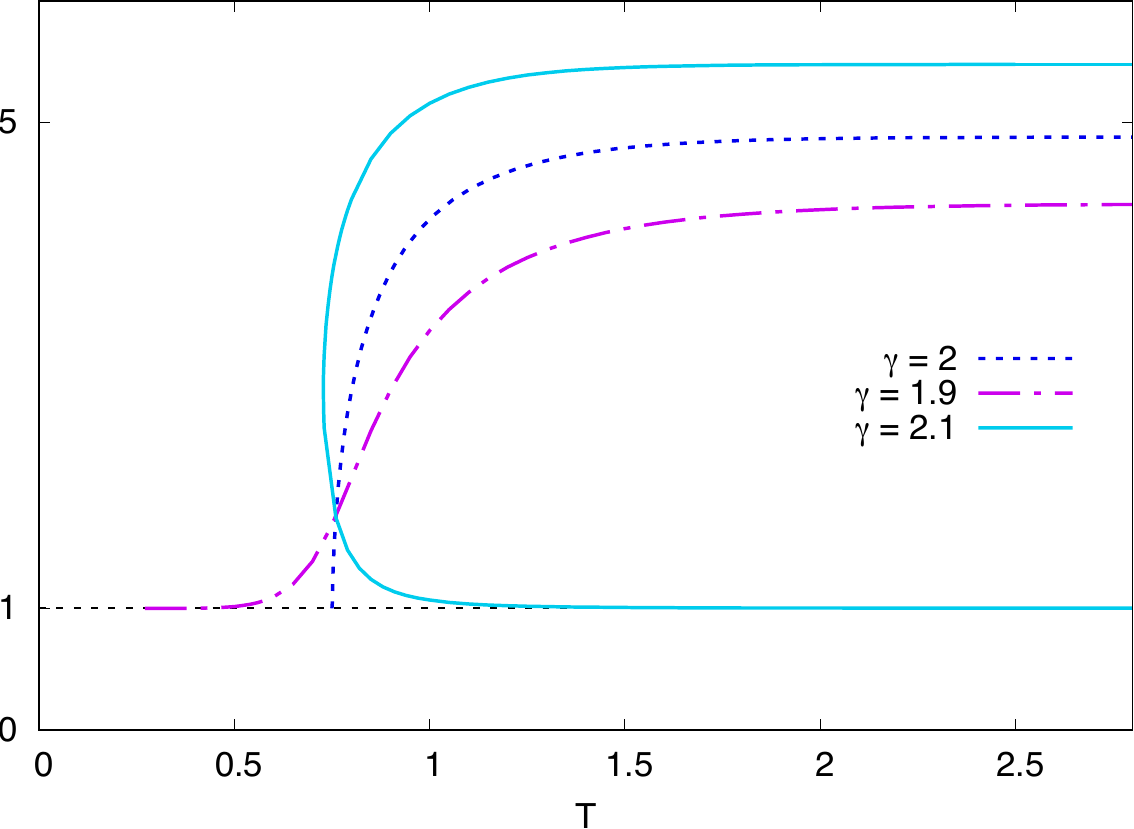}
\caption{\footnotesize Comparison between branches obtained with different values of $\gamma$.}\label{nonqbranches}
\end{figure}

\subsection{A two-populations model}\label{spb}

We consider now a truly multi-population model, introduced in \cite{ABC}. Inspired by the pioneering work of T. Schelling, the authors develop in the MFG framework a simple model of residential choice. In this model, preferences of each player are described by couplings $V_i$ of the following form:
\[
V_1(m_1, m_2) = K_1 \left(\frac{m_1}{m_1 + m_2} - \alpha_1 \right)^-, \qquad V_2(m_1, m_2) = K_2 \left(\frac{m_2}{m_1 + m_2} - \alpha_2 \right)^-,
\]
where $(\cdot)^-$ denotes the usual negative part function, $K_i > 0$, and $\alpha_i \in [0,1]$. In particular, $\alpha_i$ represents the minimum percentage of players of the $i$-th population among the total amount of players that is required for an agent of the $i$-th population to be satisfied. In other words, if the ratio between $m_1(x)$ and $m_1(x) + m_2(x)$ is above the threshold $\alpha_1$, players of the first population at position $x$ pay a null cost; otherwise they pay a positive cost, and might be tempted to move to another spot (the scenario is identical for players of the second population). We refer to \cite{ABC} for additional details regarding this model. What we aim to show here is the existence of solutions with an instable-oscillatory behaviour, that has been pointed out numerically in the aforementioned work when thresholds $a_i$ are larger than $1/2$. 

Before presenting numerical experiments, let us set this model into the bifurcation framework developed in Section \ref{s_bif}. We first note that the negative part function is not $C^1$ in a neighborhood of zero, but it is sufficient to replace it by a regularized version that coincides with $(\cdot)^-$ in $\R \setminus (-\eta, \eta)$, with $\eta > 0$ very small. We may now proceed to compute $JV$, and to evaluate it at $(1,1)$. Since
\[
\begin{split}
& (\partial_{m_1} V_1(m_1, m_2), \partial_{m_2} V_1(m_1, m_2)) = 
\begin{cases}
K_1\left(-\frac{m_{2}}{(m_1 + m_2)^2}, \frac{m_1}{(m_1 + m_2)^2} \right) & \text{if $m_1(1-\alpha_1) - \alpha_1 m_{2} < 0$} \\
(0, 0) & \text{if $m_1(1-\alpha_1) - \alpha_1 m_{2} > 0$}
\end{cases}\\
& (\partial_{m_1} V_2(m_1, m_2), \partial_{m_2} V_2(m_1, m_2)) = 
\begin{cases}
K_2\left(\frac{m_2}{(m_1 + m_2)^2}, -\frac{m_{1}}{(m_1 + m_2)^2} \right) & \text{if $m_2(1-\alpha_2) - \alpha_2 m_{1} < 0$} \\
(0, 0) & \text{if $m_2(1-\alpha_2) - \alpha_2 m_{1} > 0$}
\end{cases}
\end{split}
\]
we basically distinguish three cases:

 {\bf 1).} Both $\alpha_1, \alpha_2 < 1/2$, so $JV(1,1) = 0$. In this regime, the bifurcation results do not apply, because the linearized system does not have non-trivial solutions; we then expect the trivial solution to be isolated for all $T$. This is reasonable as both $V_1(\bar m)$ and $V_2(\bar m)$ are identically zero, so both the populations are completely satisfied in the constant-trivial state; any variation or movement would increase the cost.

{\bf 2).} Both $\alpha_1, \alpha_2 > 1/2$, so 
\[
\begin{bmatrix}
-\frac{K_1+K_2}{4} &  0 \\
  0 & 0
\end{bmatrix}
= \Xi \cdot JV \cdot \Xi^{-1}, \qquad
\Xi = \begin{bmatrix}
K_1 &  1 \\
 -K_2 &  1
\end{bmatrix},
\]
and Theorem \ref{glob_bifo_mfg} applies whenever $\sigma^2 \lambda_1 < \frac{K_1+K_2}{4} < \sigma^2 \lambda_2$ (see also Remark \ref{addrem1} for larger values of $K_1 + K_2$). Branches of non-trivial solutions exist and originate at $T^*_n$ given by \eqref{TC}; note that, close to bifurcation points, we have the representation
\[
(m_1(x,s), m_2(x,s)) = (1 + \eps K_1 \psi_1(x) \sin \left(2 \pi s/\tau \right), 1 - \eps K_2 \psi_1(x) \sin \left(2 \pi s/\tau \right))  + o(\eps)
\]
where $\tau = 2\pi \left(\lambda_{1}((K_1+K_2)/4 - \sigma^2 \lambda_{1})\right)^{-1/2}$. 

The structure of non-trivial branches is somehow similar to the single population model with aggregation discussed in Section \ref{spa}. This is evident close to bifurcation points, since the multi-population linarized system can be decoupled via $\Xi$ in two systems: the first one coincides with the one that is obtained by linearizing the single population case with aggregation, while the second one never has non-trivial solutions, so it does not generate bifurcation points. The oscillating behaviour in the Schelling's model is then analogous to the one that is observed in Section \ref{spa}, see Experiment 4.

 {\bf 3).} $\alpha_1 > 1/2$ while $\alpha_2 < 1/2$, so the second population (more tolerant) is happy in the trivial equilibrium, while players of the first population pay a positive cost. Here, 
\[
\begin{bmatrix}
-\frac{K_1}{4} &  0 \\
  0 & 0
\end{bmatrix}
= \Xi \cdot JV \cdot \Xi^{-1}, \qquad
\Xi = \begin{bmatrix}
1 &  1 \\
0 &  1
\end{bmatrix},
\]
so Theorem \ref{glob_bifo_mfg} applies if $\sigma^2 \lambda_1 < K_1/4 < \sigma^2 \lambda_2$ (see also Remark \ref{addrem1} for larger values of $K_1$). Non-trivial solutions branch off from the trivial ones at some $T^*_n$, and close to bifurcation points we have the representation
\begin{equation}\label{locpar}
(m_1(x,s), m_2(x,s)) = (1 + \eps\psi_1(x) \sin \left(2\pi s/\tau \right), 1)  + o(\eps)
\end{equation}
where $\tau = 2\pi \left(\lambda_{1}(K_1/4 - \sigma^2 \lambda_{1})\right)^{-1/2}$. This regime is somehow more peculiar than the previous setting 2), where both populations are intolerant. In that case, close to bifurcation points we expect the first population to oscillate approximately between $1+  \eps \psi$ and $1- \eps \psi$, while the other one to oscillate between $1- \eps \psi$ and $1+\eps \psi$, trying to avoid each other. Here, small oscillations of the first population are not strong enough to induce the first population to leave the constant state, at least up to some critical $T$. On the other hand, if the perturbation from the trivial state of $m_1$ is significant, $m_2$ will oscillate itself to avoid $m_1$ and decrease its own cost. Far from bifurcation points, oscillations become significant for both populations. See Experiment 5 for additional considerations on this regime.

\smallskip \noindent {\bf Experiment 4.} The parameters are chosen as follows
\[
\sigma = \frac{1}{\pi}, \quad a_1 = 0.7, \quad a_2 = 0.55, \quad K_1 = 5, \quad K_2 = 3.
\]
Players of both populations prefer their spot to be occupied by players of their own population, but the first one is somehow more racist than the other. In Figure \ref{fig_sch} we observe the typical oscillating behaviour close to and far from bifurcation points. Populations try to avoid each other; while the first branch of equilibria consists of two profiles for $m_1$ and $m_2$ respectively that are stable in time, switching between different states arises in other branches; still, the number of switchings appears to be preserved in every branch. We observe that the population that is more tolerant is more ``spread'' on $[0,1]$, while the other one is more concentrated, i.e. its maximum is larger.

\begin{figure}
\centering
\includegraphics[width=4cm]{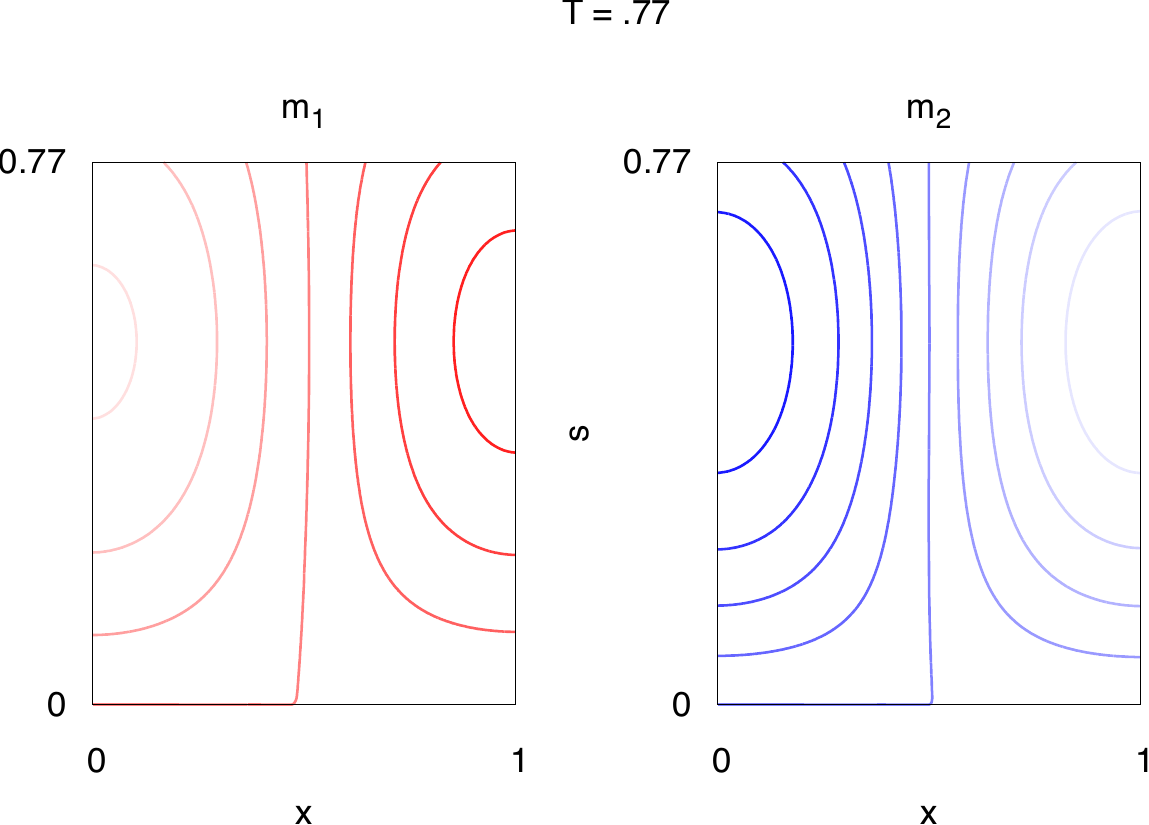}
\hspace{.5cm}
\includegraphics[width=4cm]{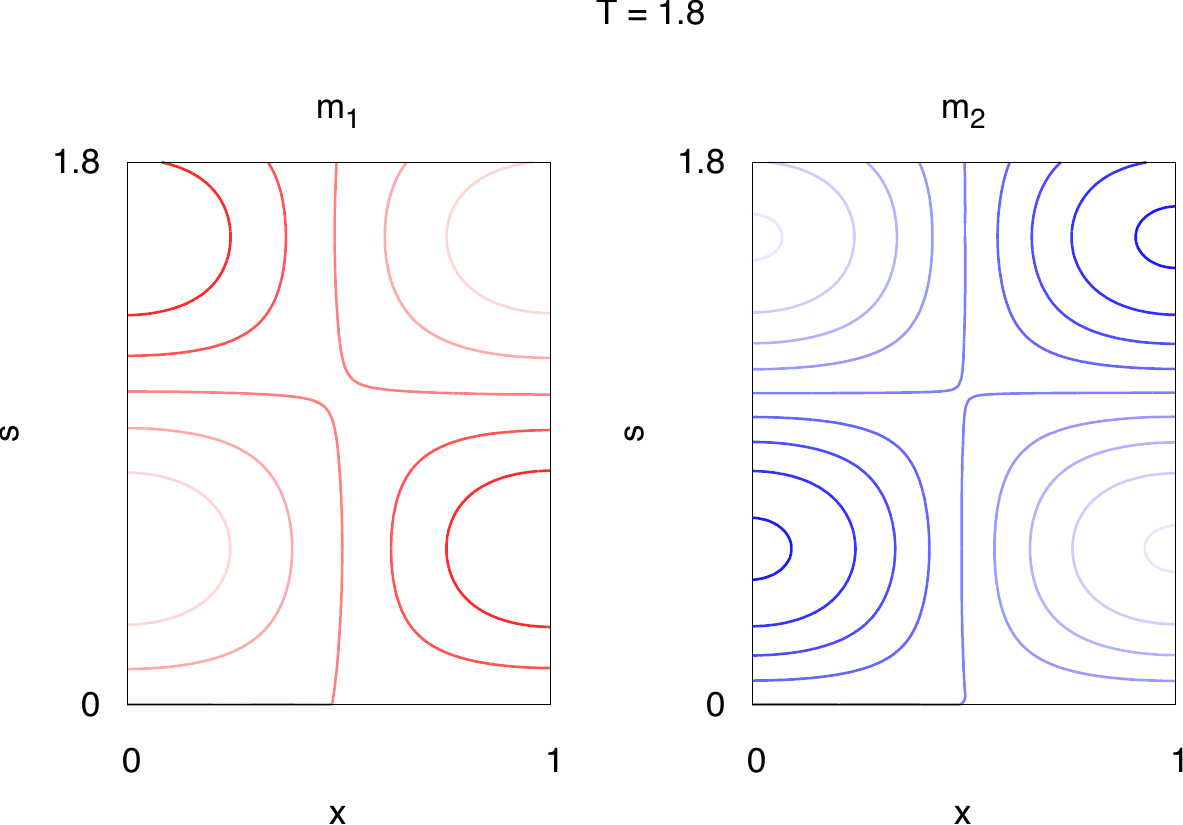}
\hspace{.5cm}
\includegraphics[width=4cm]{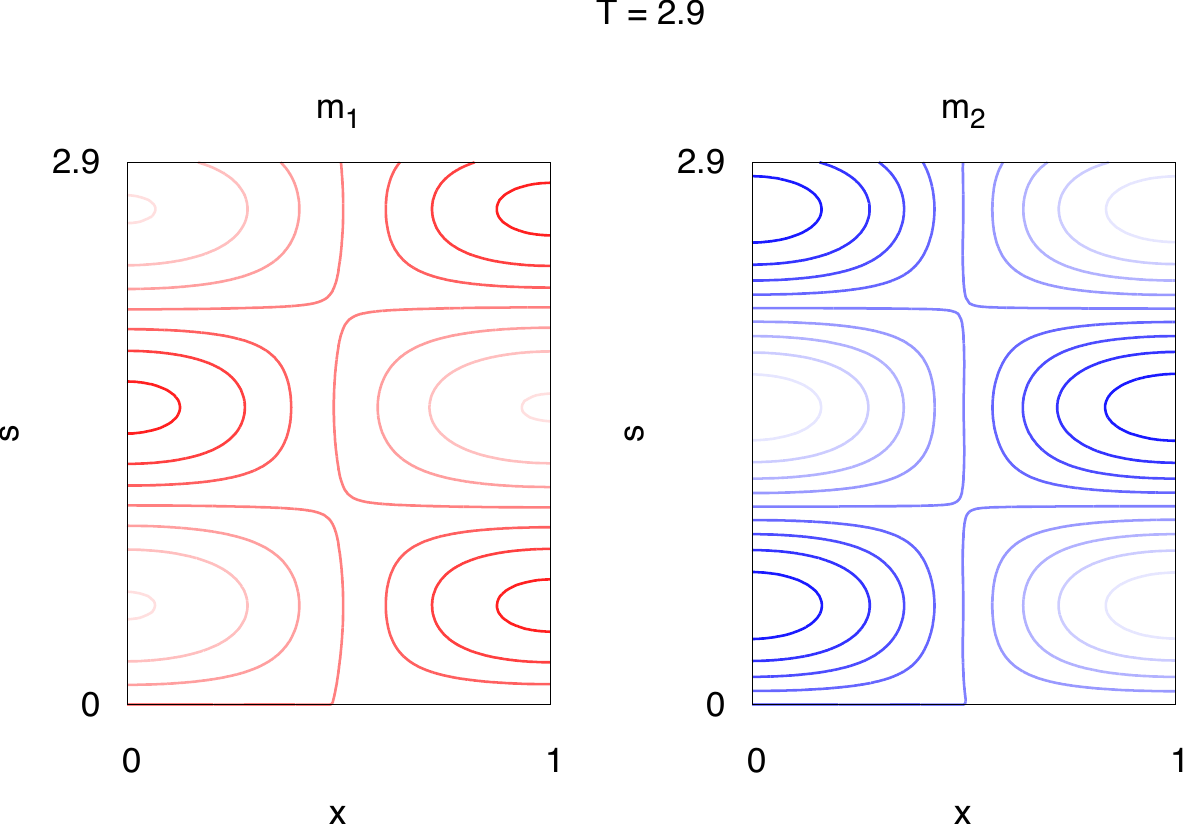}
\includegraphics[width=4cm]{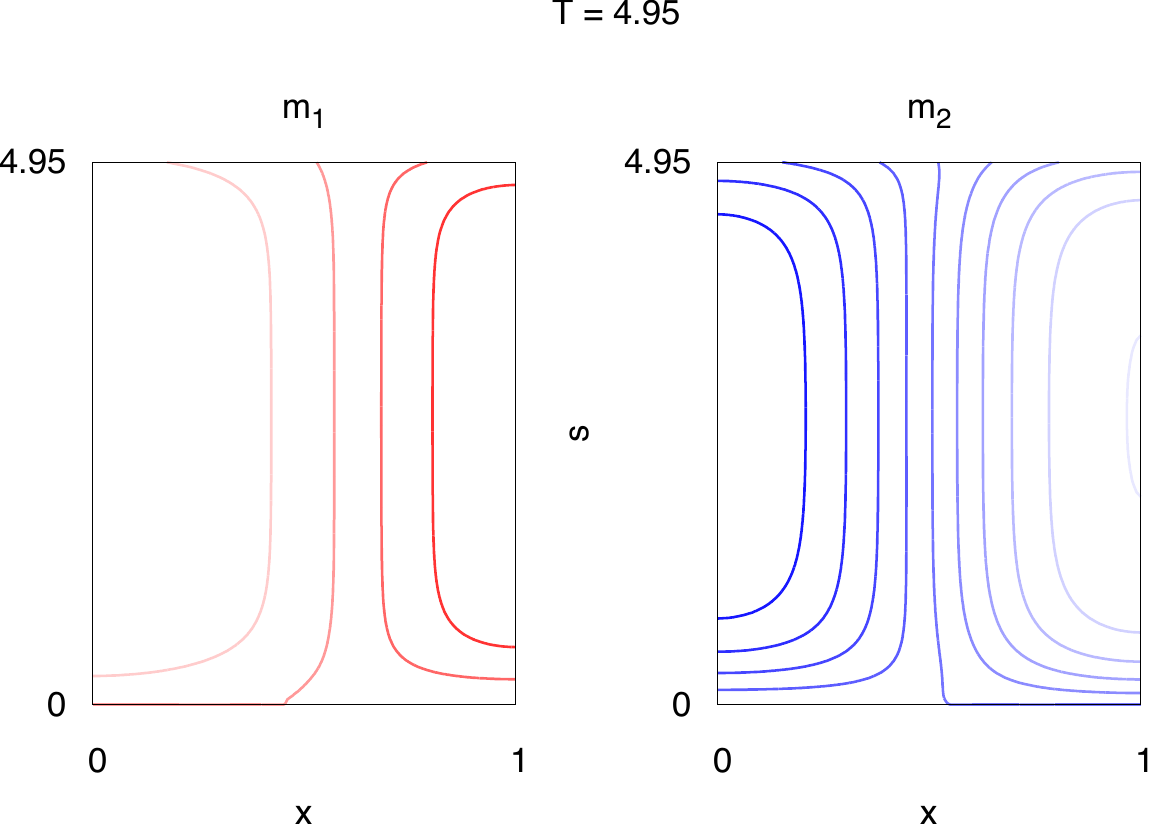}
\hspace{.5cm}
\includegraphics[width=4cm]{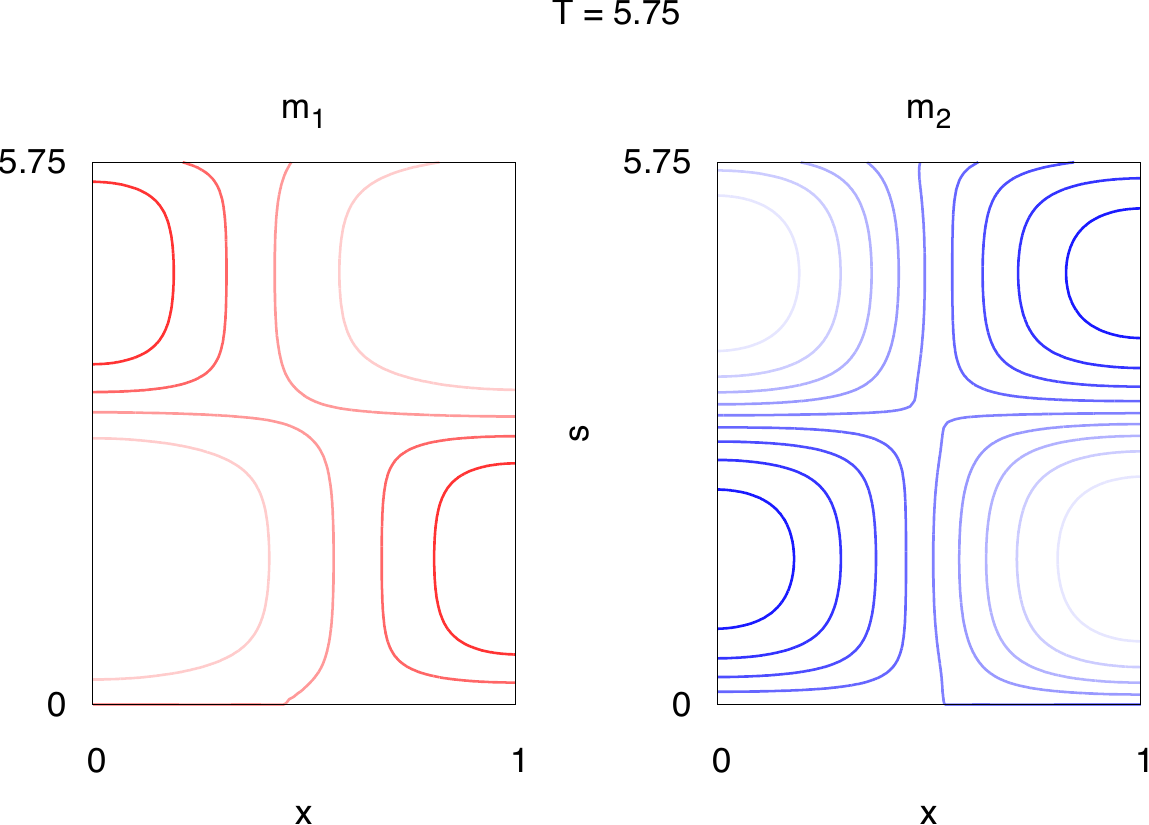}
\hspace{.5cm}
\includegraphics[width=4cm]{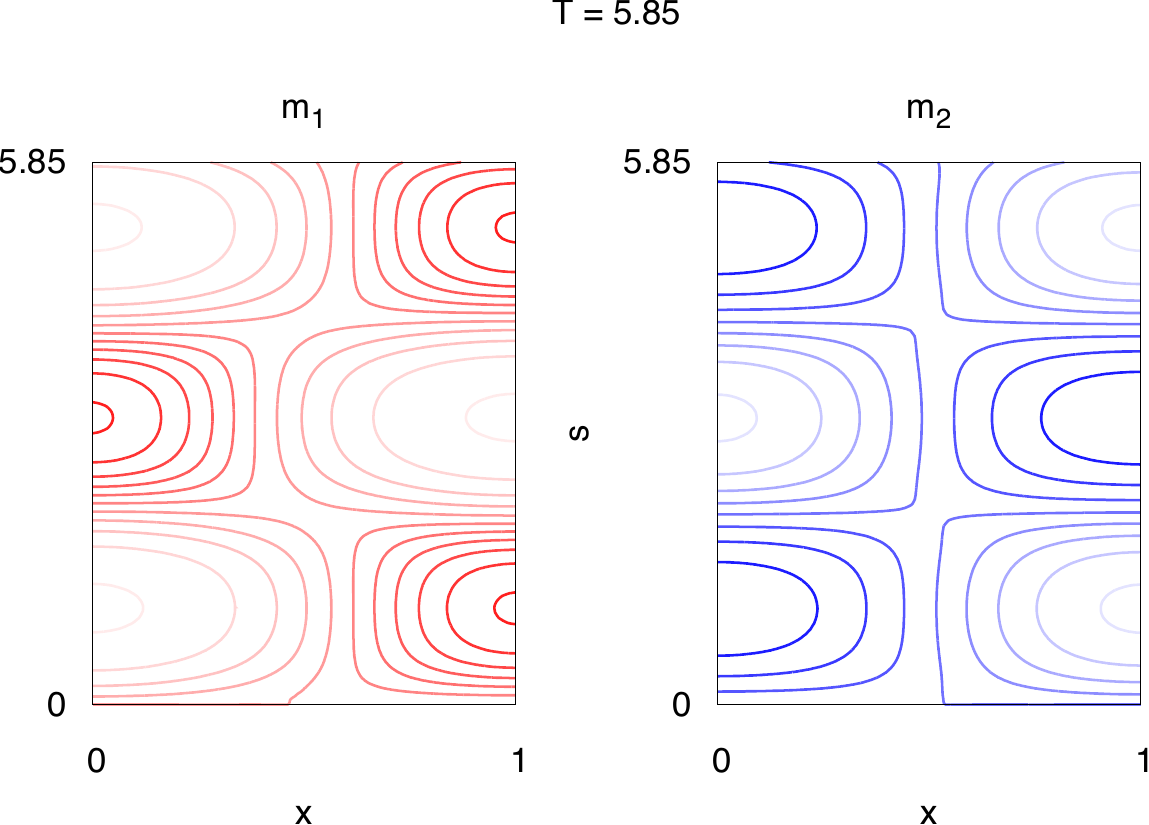}

\caption{\footnotesize Space-time contours of $m_1, m_2$ belonging to different branches. Top contours are for solutions close to bifurcation times, while bottom ones are taken for bigger $T$.}\label{fig_sch}
\end{figure}

\smallskip \noindent {\bf Experiment 5.} The parameters are chosen as follows
\[
\sigma = \frac{1}{\pi}, \quad a_1 = 0.8, \quad a_2 = 0.4, \quad K_1 = 8, \quad K_2 = 8.
\]
With respect to Experiment 4, we point out the following phenomenon. Recall that, close to bifurcation points, solutions are parametrized by \eqref{locpar}. Since $a_2 < 0.5$, if $\eps$ is small enough, $V_2(m_1(x,s), m_2(x,s))$ is identically zero; this implies that $(u_2, m_2)$ in \eqref{MFG0} must be the trivial couple $(0,1) $ on $Q_T$, and in turn, \eqref{locpar} becomes
\[
(m_1(x,s), m_2(x,s)) = (1 + \eps\psi_1(x) \sin \left(\omega s \right) + o(\eps), 1) 
\]
In other words, close to a bifurcation point $T^*$, it is not convenient for $m_2$ to leave the constant state. This is particularly evident in Figure \ref{fig_sch_br} (see the blue line). As soon as $T$ increases, $\|m_1 - 1\|_\infty$ increases, so  $V_2(m_1(x,s), m_2(x,s))$ becomes non-zero on $Q_T$. At this point, the behaviour of the branch changes abruptly: it reaches a turning point (vertical dashed line in Figure \ref{fig_sch_br}), after which $m_2$ becomes truly non-trivial. Following the branch, another turning point is reached, and the qualitative behaviour of solutions then mimic the one of Experiment 4.

\begin{figure}
\centering
\includegraphics[width=6cm]{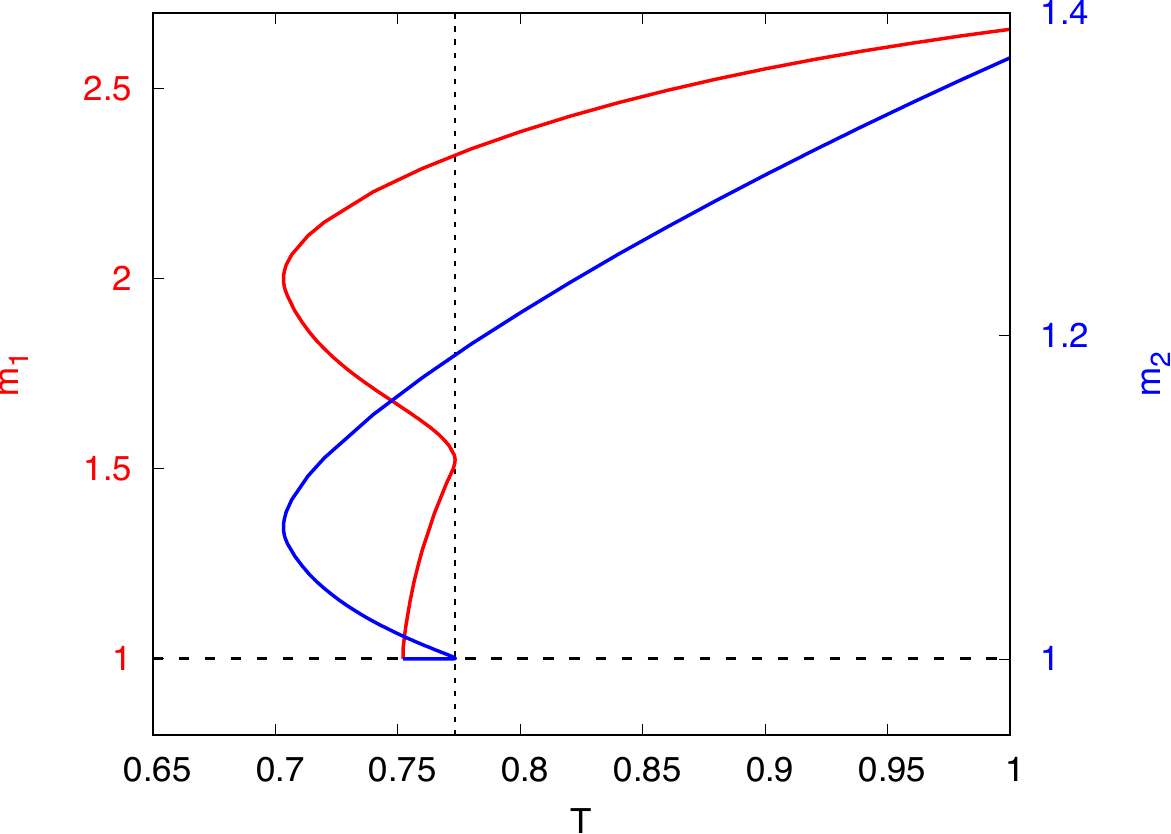}

\caption{\footnotesize The first branch of non-trivial solutions in Experiment 5. Red line is the maximum value $\|m_1\|_{\infty}$ vs. $T$, while the blue line represents $\|m_2\|_{\infty}$.} \label{fig_sch_br}
\end{figure}


\medskip
\begin{flushright}
\noindent \verb"cirant@math.unipd.it"\\
Dipartimento di Matematica ``Tullio Levi-Civita''\\ Universit\`a di Padova\\
via Trieste 63, 35121 Padova (Italy)
\end{flushright}


\small

\begin{thebibliography}{10}

\bibitem{ABC}
Y.~Achdou, M.~Bardi, and M.~Cirant.
\newblock Mean field games models of segregation.
\newblock {\em Math. Models Methods Appl. Sci.}, 27(1):75--113, 2017.

\bibitem{BCir}
M.~Bardi and M.~Cirant.
\newblock Uniqueness of solutions in mean field games with several populations
  and neumann conditions.
\newblock arXiv preprint, https://arxiv.org/abs/1709.02158, 09 2017.

\bibitem{BarFis}
M.~Bardi and M.~Fischer.
\newblock On non-uniqueness and uniqueness of solutions in finite-horizon mean
  field games.
\newblock {\em to appear in ESAIM Control Optim. Calc. Var.}, 2018.

\bibitem{BenFr}
A.~Bensoussan, J.~Frehse, and P.~Yam.
\newblock {\em Mean field games and mean field type control theory}.
\newblock SpringerBriefs in Mathematics. Springer, New York, 2013.

\bibitem{Lauriere}
A.~Bensoussan, T.~Huang, and M.~Lauri\'ere.
\newblock Mean field control and mean field game models with several
  populations.
\newblock {\em Minimax Theory Appl.}, 3:173--209, 2018.

\bibitem{BriCar}
A.~Briani and P.~Cardaliaguet.
\newblock Stable solutions in potential mean field game systems.
\newblock {\em NoDEA Nonlinear Differential Equations Appl.}, 25(1):Art. 1, 26,
  2018.

\bibitem{CardaNotes}
P.~Cardaliaguet.
\newblock Notes on mean field games.

\bibitem{CLLP}
P.~Cardaliaguet, J.-M. Lasry, P.-L. Lions, and A.~Porretta.
\newblock Long time average of mean field games.
\newblock {\em Netw. Heterog. Media}, 7(2):279--301, 2012.

\bibitem{Carda17}
P.~Cardaliaguet, A.~Porretta, and D.~Tonon.
\newblock {\em A Segregation Problem in Multi-Population Mean Field Games},
  pages 49--70.
\newblock Springer International Publishing, Cham, 2017.

\bibitem{CCes}
A.~Cesaroni and M.~Cirant.
\newblock Concentration of ground states in stationary mean-field games
  systems.
\newblock {\em Anal. PDE}, 12(3):737--787, 2019.

\bibitem{C14}
M.~Cirant.
\newblock Multi-population mean field games systems with {N}eumann boundary
  conditions.
\newblock {\em J. Math. Pures Appl. (9)}, 103(5):1294--1315, 2015.

\bibitem{cir16}
M.~Cirant.
\newblock Stationary focusing mean-field games.
\newblock {\em Comm. Partial Differential Equations}, 41(8):1324--1346, 2016.

\bibitem{CGM}
M.~Cirant, R.~Gianni, and P.~Mannucci.
\newblock Short-time existence for a backward-forward parabolic system arising
  from mean-field games.
\newblock arXiv:1806.08138, 2018.

\bibitem{CNur}
M.~Cirant and L.~Nurbekyan.
\newblock The variational structure and time-periodic solutions for mean-field
  games systems,.
\newblock {\em Minimax Theory Appl.}, 3:227--260, 2018.

\bibitem{CTon}
M.~Cirant and D.~Tonon.
\newblock Time-dependent focusing mean-field games: the sub-critical case.
\newblock arXiv:1704.04014, to appear in J. Dynam. Differential Equations,
  2017.

\bibitem{CV}
M.~Cirant and G.~Verzini.
\newblock Bifurcation and segregation in quadratic two-populations mean field
  games systems.
\newblock {\em ESAIM Control Optim. Calc. Var.}, 23(3):1145--1177, 2017.

\bibitem{feleqi}
E.~Feleqi.
\newblock The derivation of ergodic mean field game equations for several
  populations of players.
\newblock {\em Dyn. Games Appl.}, 3(4):523--536, 2013.

\bibitem{GS}
D.~Gomes and M.~Sedjro.
\newblock One-dimensional, forward-forward mean-field games with congestion.
\newblock {\em Discrete Contin. Dyn. Syst. Ser. S}, 11(5):901--914, 2018.

\bibitem{Go16}
D.~A. Gomes, L.~Nurbekyan, and M.~Prazeres.
\newblock One-dimensional stationary mean-field games with local coupling.
\newblock {\em Dyn. Games Appl.}, 8(2):315--351, 2018.

\bibitem{GoPa}
D.~A. Gomes and S.~Patrizi.
\newblock Weakly coupled mean-field game systems.
\newblock {\em Nonlinear Anal.}, 144:110--138, 2016.

\bibitem{GomesBook}
D.~A. Gomes, E.~A. Pimentel, and V.~Voskanyan.
\newblock {\em Regularity theory for mean-field game systems}.
\newblock SpringerBriefs in Mathematics. Springer, 2016.

\bibitem{HCM06}
M.~Huang, R.~P. Malham\'e, and P.~E. Caines.
\newblock Large population stochastic dynamic games: closed-loop
  {M}c{K}ean-{V}lasov systems and the {N}ash certainty equivalence principle.
\newblock {\em Commun. Inf. Syst.}, 6(3):221--251, 2006.

\bibitem{Kiel}
H.~Kielh\"ofer.
\newblock {\em Bifurcation theory}, volume 156 of {\em Applied Mathematical
  Sciences}.
\newblock Springer-Verlag, New York, 2004.
\newblock An introduction with applications to PDEs.

\bibitem{LachapelleWolfram}
A.~Lachapelle and M.-T. Wolfram.
\newblock On a mean field game approach modeling congestion and aversion in
  pedestrian crowds.
\newblock {\em Transp. Res. Part B: Methodol.}, 45(10):1572 -- 1589, 2011.

\bibitem{LL062}
J.-M. Lasry and P.-L. Lions.
\newblock Jeux \`a champ moyen. {II}. {H}orizon fini et contr\^ole optimal.
\newblock {\em C. R. Math. Acad. Sci. Paris}, 343(10):679--684, 2006.

\bibitem{LL07}
J.-M. Lasry and P.-L. Lions.
\newblock Mean field games.
\newblock {\em Jpn. J. Math.}, 2(1):229--260, 2007.

\bibitem{Lieber}
G.~M. Lieberman.
\newblock {\em Second order parabolic differential equations}.
\newblock World Scientific Publishing Co., Inc., River Edge, NJ, 1996.

\bibitem{Lcol}
P.-L. Lions.
\newblock In cours au coll\'ege de france.
\newblock {\em www.college-de-france.fr}.

\bibitem{YMMS}
H.~Yin, P.~G. Mehta, S.~P. Meyn, and U.~V. Shanbhag.
\newblock Bifurcation analysis of a heterogeneous mean-field oscillator game
  model.
\newblock In {\em Proceedings of the 50th {IEEE} Conference on Decision and
  Control and European Control Conference, {CDC-ECC} 2011}, pages 3895--3900,
  2011.

\end{thebibliography}
\end{document}